\documentclass[a4paper,11pt,reqno]{amsart}
\usepackage{graphicx, psfrag, amsmath,  amssymb, array, hyperref}
\usepackage{xcolor,colortbl} 
\usepackage{enumitem, cleveref}

\newtheorem{lem}{Lemma}[section]
\newtheorem{thm}[lem]{Theorem}
\newtheorem{pro}[lem]{Proposition}
\newtheorem{cor}[lem]{Corollary}
\newtheorem{exa}[lem]{Example}

\newtheorem{con}[lem]{Conjecture}
\newtheorem{defi}[lem]{Definition}

\theoremstyle{definition}
\newtheorem{algo}{Algorithm}
\renewcommand{\thealgo}{\Alph{algo}}

\numberwithin{equation}{section}

\newcommand{\ZZ}{{\mathbb{Z}}}

\newcommand{\A}{{\mathcal{A}}}

\newcommand{\Des}{{\textsf{Des}}}

\newcommand{\Asc}{{\textsf{Asc}}}

\newcommand{\stem}{{\mathsf{stem}}}
\newcommand{\tail}{{\mathsf{tail}}}
\newcommand{\branch}{{\mathsf{branch}}}
\newcommand{\cc}{{\mathrm{c}}}

\newcommand{\MF}{{\mathcal{MF}}}
\newcommand{\CF}{{\mathcal{CF}}}
\newcommand{\TF}{{\mathcal{TF}}}

\newcommand{\LT}{{\mathcal{LT}}}
\newcommand{\QT}{{\mathcal{QT}}}
\newcommand{\SF}{{\mathcal{SF}}}
\newcommand{\OSF}{{\mathcal{OSF}}}

\newcommand{\F}{{\mathcal{F}}}

\newcommand{\T}{{\mathcal{T}}}

\newcommand{\odd}{{\textrm{odd}}}
\newcommand{\even}{{\textrm{even}}}

\title[]{On Ternary Trees and Fighting Fish}

\author[S.-P. Eu]{Sen-Peng Eu}
\address{Department of Mathematics, National Taiwan Normal University, Taipei 116325, and Chinese Air Force Academy, Kaohsiung 820009, Taiwan, ROC}
\email{speu@math.ntnu.edu.tw}

\author[T.-S. Fu]{Tung-Shan Fu}
\address{Department of Applied Mathematics, National Pingtung University, Pingtung 900391, Taiwan, ROC}
\email{tsfu@mail.nptu.edu.tw}

\author[Y.-R. Pan]{Yu-Ren Pan}
\address{Department of Mathematics, National Taiwan Normal University, Taipei 116325, Taiwan, ROC}
\email{yuren.pan.0117@gmail.com}

\begin{document}

\begin{abstract}
Fighting fish is a combinatorial configuration introduced by Duchi, Guerrini, Rinaldi and Schaeffer as a new model of branching surfaces that generalizes directed convex  polyominoes. We come up with an alternative construction of fighting fish, using a tree structure built on the so-called stem cells of fighting fish. From this perspective, we establish a bijection between ternary trees and fighting fish with a marked strip of cells, which specializes to a direct bijection between left ternary trees and fighting fish. Using these results, we obtain a combinatorial enumeration of the fighting fish of size $n$ by establishing an $(n+1)$-to-2 bijection with the ternary trees with $n$ nodes. We present some additional enumerative results including that fighting fish with a marked tail and horizontally symmetric fighting fish are equinumerous with ordered pairs of ternary trees having a total of a given number of nodes.

\end{abstract}

\maketitle

\section{Introduction} \label{sec:Introduction}

In a study of random branching surfaces, Duchi, Guerrini, Rinaldi and Schaeffer \cite{DGRS-A} introduced a combinatorial configuration called fighting fish in terms of gluings of unit squares as a  generalization of directed convex polyominoes. Fighting fish are counted by the sequence \cite[A000139]{oeis} 
\begin{equation*}
\left\{\frac{2}{(n+1)(2n+1)}\binom{3n}{n}\right\}_{n\ge 1} =  1,2,6,22,91,408,1938,9614,\dots,
\end{equation*}
which turns out to be the number of some classical objects such as left ternary trees \cite{DLDRP, JS-1998}, two-stack-sortable permutations \cite{CFH, Fang-2018-2, GW-1996, West-1993, Zeil-1992}, non-separable rooted planar maps \cite{Brown, Brown-Tutte, DH-2022-A, FP-2017} and synchronized intervals of the Tamari lattice \cite{DH-2023, FP-2017}. 

In recent years, there has been growing interest in bijections among these different families of objects. 
Goulden and West \cite{GW-1996} established bijections between two-stack-sortable permutations and non-separable rooted planar maps using either generating trees or recursive decompositions (see also \cite{DGW}). Del Lungo, Del Ristoro and Penaud \cite{DLDRP} used an analogous decomposition for left ternary trees to establish a bijection between non-separable rooted planar maps and left ternary trees, whereas Jacquard and Schaeffer \cite{JS-1998} discovered an alternative bijection.
Fang \cite{Fang-2018-2} used another recursive decomposition of two-stack-sortable permutations to establish a bijection with fighting fish, whereas Cioni, Ferrari and Henriet \cite{CFH} described a direct construction of Fang's bijection using a particular class of trees. Duchi and Henriet obtained a bijection between non-separable rooted planar maps and fighting fish \cite{DH-2022-A} and a bijection between synchronized intervals of the Tamari lattices and fighting fish \cite{DH-2023}. 
See Figure \ref{fig:diagram} for an updated version of a diagram, previously appeared in the article \cite{DH-2022-A}, which summarizes these bijective results (the dashed arrows indicate recursive bijections). 
One of our main results is a direct bijection between left ternary trees and fighting fish (see Theorem \ref{thm:LTT-to-fish-bijection}), contributing a missing link in the diagram.

\begin{figure}[ht]
\begin{center}

\includegraphics[width=5.2in]{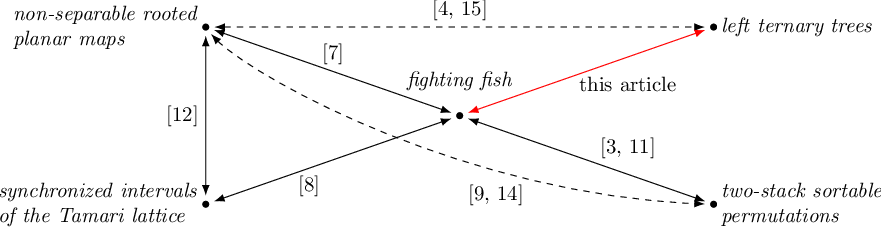}
\end{center}
\caption{\small A diagram of bijective results for the objects in connection to fighting fish.} 
\label{fig:diagram}
\end{figure}

\subsection{Fighting fish}

A \emph{cell} is a 45$^\circ$ tiled unit square the four edges of which are distinguished and called \emph{left upper edge}, \emph{left lower edge}, \emph{right upper edge} and \emph{right lower edge}. An edge of a cell is \emph{free} if it is not glued to the edge of another cell. A \emph{fighting fish} is a finite set of cells glued together edge by edge that can be constructed from an initial cell called \emph{head} by attaching new cells successively. More precisely, there are three ways to add a new cell (see Figure \ref{fig:cell-gluing}).

\setcounter{algo}{1}

\begin{enumerate}[label=(\thealgo\arabic*), ref=(\thealgo\arabic*)]
\item \emph{Upper gluing}: we attach it to a free right upper edge of a cell. \label{enu:A1}
\item \emph{Lower gluing}: we attach it to a free right lower edge of a cell. \label{enu:A2}
\item \emph{Double gluing}: if there is a cell $c$ with two cells $a$ and $b$ attached to its right upper and lower edge, and such that $a$ (resp. $b$) has a free right lower (resp. upper) edge, then we attach the new cell to both $a$ and $b$. \label{enu:A3}
\end{enumerate}

\begin{figure}[ht]
\begin{center}
\psfrag{a}[][][1]{$a$}
\psfrag{b}[][][1]{$b$}
\psfrag{c}[][][1]{$c$}
\includegraphics[width=4.8in]{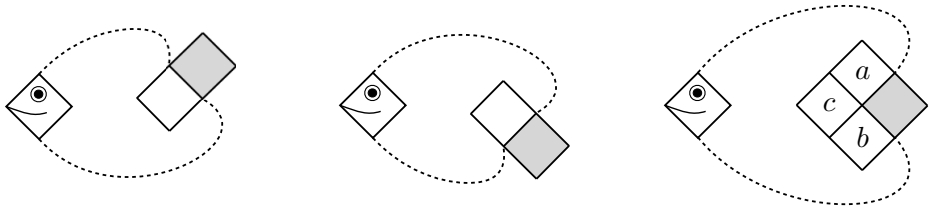}
\end{center}
\caption{\small The three operations to add a cell for growing a fighting fish.} \label{fig:cell-gluing}
\end{figure}

The head of a fighting fish is the only cell with two free left edges. A \emph{tail} is a cell with two free right edges.  
The \emph{size} of a fighting fish is the number of its free right edges minus 1.  Note that every fighting fish has as many free right edges as free left edges. Let $\F_n$ denote the set of fighting fish of size $n$. Using a classical method of generating function, Duchi, Guerrini, Rinaldi and Schaeffer \cite{DGRS-A} derived the number of fighting fish of size $n$, i.e., for $n\ge 1$, 
\begin{equation} \label{eqn:FF_number}
|\F_n| = \frac{2}{(n+1)(2n+1)}\binom{3n}{n}.
\end{equation}

\subsection{Ternary trees}
A \emph{ternary tree} is a finite set of nodes which is either empty, or contains a root and three disjoint ternary trees known as the left, middle and right subtrees of the root.
For each node $x$ of a ternary tree $T$,
there is a path $\mu(x)$ from the root to the node $x$. The path $\mu(x)$ consists of three kinds of steps, left, middle and right. The \emph{abscissa} of $x$ is the difference between the number of 
left and right steps in $\mu(x)$. More precisely, the root has abscissa 0, and if $x$ is an internal node with abscissa $j$, the left, middle and right child of $x$ has abscissa $j+1,j$ and $j-1$, respectively.  By convention, we orient the abscissa axis toward the left. Let $\alpha(x)$ denote the abscissa of $x$. Let $\T_n$ denote the set of ternary trees with $n$ nodes. It is a classical result \cite[A001764]{oeis} that
\begin{equation} \label{eqn:ternary_tree_number}
|\T_n|=\frac{1}{2n+1}\binom{3n}{n}.
\end{equation}
The generating function for $|\T_n|$, $G(x):=\sum_{n\ge 0} |\T_n|x^n = 1+x+3x^2+12x^3+55x^4+273x^5+\cdots$, satisfies the equation $G(x)=1+xG(x)^3$.
A \emph{left ternary tree} is a ternary tree whose nodes all have nonnegative abscissas. Let $\LT_n$ denote the set of left ternary trees with $n$ nodes. 

\subsection{Main results} 
In Section \ref{sec:tree-structure}, we introduce the notion of \emph{stem cells} of fighting fish (see Definition \ref{def:branch-cell}) and a tree structure built on stem cells (see Lemma \ref{lem:TF-is-tree}). In Section \ref{sec:marked-cell}, we devise a vertex-labeling scheme for ternary trees and present an alternative construction of fighting fish, which establishes a bijection that gives an interpretation of the relation 
\begin{equation} \label{eqn:relation}
\sum_{n\ge 1} n|\F_n|x^n = xG(x)^4.
\end{equation}

The \emph{conjugation} of a fighting fish $F$, denoted by $F^{\cc}$, is the fighting fish obtained from $F$ by reflecting with respect to the horizontal axis. For a cell $t$ of $F$, let $t^{\cc}$ denote the cell in $F^{\cc}$ which is the reflection image of the cell $t$. For all $n\ge 1$, we define
\begin{equation} \label{eqn:fish-cells}
\CF_n =\{(F,t)\mid F\in\F_n \mbox{ and $t$ is a stem cell of $F$}\}.
\end{equation}

Let $\T=\cup_{n\ge 0} \T_n$. For all $n\ge 0$, let $(\T\times \T\times \T\times \T)_n$ denote the set of ordered quadruples $(T_1,T_2,T_3,T_4)$ of ternary trees having a total of $n$ nodes. Let $\QT_n$ denote the set of ordered trees with $n$ nodes such that the root has four ternary trees as subtrees (possibly empty). Sometimes we use the notation $(t;T_1,T_2,T_3,T_4)$ for the members of $\QT_n$, where $T_1,T_2,T_3$ and $T_4$ are the subtrees of the root $t$ from left to right and $(T_1,T_2,T_3,T_4)\in(\T\times \T\times \T\times \T)_{n-1}$.

\begin{thm} \label{thm:fighting-fish-marked-stem-cell}
For all $n\ge 1$, there is a bijection $\varphi:\QT_n\rightarrow\CF_n$ such that $(t;T_1,T_2,T_3,T_4)\mapsto (F,t)$ implies $(t;T_2,T_1,T_4,T_3)\mapsto (F^{\cc},t^{\cc})$.
\end{thm}

By a \emph{descending} (resp. \emph{ascending}) \emph{strip} of a fighting fish we mean a maximal sequence $a_1, a_2, \dots, a_m$ of cells such that the left upper (resp. lower) edge of $a_{j+1}$ is glued to the right lower (resp. upper) edge of $a_j$, for some positive integer $m$ and $1\le j\le m-1$. The cell $a_1$ is the \emph{top} (resp. \emph{bottom}) cell of the descending (resp. ascending) strip, and the cell $a_m$ is the bottom (resp. top) cell of the descending (resp. ascending) strip.
In particular, the descending strip containing the head is also called the \emph{jaw} of the fighting fish. 
Note that the descending (resp. ascending) strips of $F$ become the ascending (resp. descending) strips of $F^{\cc}$. 
Define
\begin{equation*}
\MF_n:=\{(F,Q)\mid F\in\F_n \mbox{ and $Q$ is a descending strip of $F$}\}.
\end{equation*}

In Section \ref{sec:marked-strip}, based on the bijection in Theorem \ref{thm:fighting-fish-marked-stem-cell}, we establish the following bijection.

\begin{thm} \label{thm:TT-to-marked-fish} 
For all $n\ge 1$, there is a bijection $\phi: T\mapsto (F,Q)$ of $\T_n$ onto $\MF_n$ such that a ternary tree $T$ having $\ell$ nodes with odd abscissas and $n-\ell$ nodes with even abscissas is carried to an ordered pair $(F,Q)$, where the fighting fish $F$ has $\ell+1$ descending strips and $n-\ell$ ascending strips.
\end{thm}

Using (\ref{eqn:ternary_tree_number}) and Theorem \ref{thm:TT-to-marked-fish}, we give a combinatorial proof of (\ref{eqn:FF_number}) by establishing the following bijection.

\begin{thm} \label{thm:2-to-(n+1)}
There is an $(n+1)$-to-$2$ bijection between the set $\F_n$ of fighting fish of size $n$ and the set $\T_n$ of ternary trees with $n$ nodes. Hence we have
\begin{equation} \label{eqn:fish-number}
|\F_n|=\frac{2}{n+1}|\T_n|.
\end{equation}
\end{thm}

In Section \ref{sec:left-ternary}, when the map $\phi$ in Theorem \ref{thm:TT-to-marked-fish} is restricted to left ternary trees, we obtain the following bijection.

\begin{thm} \label{thm:LTT-to-fish-bijection}
For all $n\ge 1$, there is a bijection $\psi: T\mapsto F$ of $\LT_n$ onto $\F_n$ such that a left ternary tree $T$ having $\ell$ nodes with odd abscissas, $n-\ell$ nodes with even abscissas and $k$ nodes with abscissa 0 is carried to a fighting fish $F$ having $\ell+1$ descending strips, $n-\ell$ ascending strips and a jaw with $k$ cells.
\end{thm}

We remark that this bijection is established by a direct construction involving statistic insights, which is not isomorphic to the composition of the bijection between non-separable rooted planar maps and fighting fish by Duchi and Henriet \cite{DH-2022-A} and the recursive bijection between left ternary trees and non-separable rooted planar maps by Del~Lungo, Del~Ristoro and Penaud \cite{DLDRP} or by Jacquard and Schaeffer \cite{JS-1998} (see Section \ref{sec:conclusion} for remarks).

\medskip
We present some additional enumerative results derived from Theorems \ref{thm:fighting-fish-marked-stem-cell}, \ref{thm:TT-to-marked-fish} and \ref{thm:LTT-to-fish-bijection}. 
Duchi, Guerrini, Rinaldi and Schaeffer \cite{DGRS-A} derived the total number of tails of all fighting fish in $\F_n$, which coincides with the number of ordered pairs of ternary trees having a total of $n-1$ nodes (introduced by Knuth in his lecture \cite{Knuth-2014}). For $n\ge 0$, let $(\T\times\T)_n$ be the set of ordered pairs $(T_1,T_2)$ of ternary trees having a total of $n$ nodes. It is known that $|(\T\times\T)_n|$ is counted by the sequence \cite[A006013]{oeis}
\begin{equation} \label{eqn:ordered-pair-TT}
\left\{\frac{1}{n+1}\binom{3n+1}{n}\right\}_{n\ge 0} = 1, 2, 7, 30, 143, 728, 3876,\dots.
\end{equation}
For $n\ge 1$, we obtain the following bijective result  (see Theorem \ref{thm:pair-ternary-trees})
\begin{equation}
\{(F,t)\mid F\in\F_n\mbox{ and $t$ is a tail of $F$}\}  \rightarrow (\T\times\T)_{n-1}, \label{eqn:fish-tail-ordered-pair}
\end{equation}
which gives a combinatorial proof of a result by Duchi, Guerrini, Rinaldi and Schaeffer \cite[Corollary 1]{DGRS-A}.

In Section \ref{sec:additional}, we establish a bijection between the following two sets (see Theorem \ref{thm:SFF_enumeration}) to enumerate horizontally symmetric fighting fish ($\SF$).
\begin{equation}
\SF_{2n+1} \rightarrow (\T\times\T)_n.
\end{equation}
Moreover, we obtain a combinatorial enumeration of left ternary trees (see Theorem \ref{thm:enumeration})
\begin{equation}
|\LT_n| = \frac{2}{n+1}|\T_n|,
\end{equation}
which gives a simplified proof of a result by Jacquard and Schaeffer (cf. \cite[Theorem 2]{JS-1998}).

\section{A tree structure on the stem cells of fighting fish} \label{sec:tree-structure}

In this section, we introduce the notion of stem cells of  fighting fish and present a tree structure on the stem cells. First, we have a simple observation about fighting fish.

For some $m\ge 4$, a sequence $c_1,c_2,\dots,c_m$ of cells in a fighting fish is called a \emph{cell circuit} if $c_1$ is attached to $c_m$, and $c_i$ is attached to $c_{i-1}$ for $2\le i\le m$. It is called a \emph{cell cycle} if no cell is repeated.
A cell cycle of a fighting fish is called a \emph{hole} if the edges enclosed within the simple closed curve connecting the centers of consecutive cells along the cycle are free. 

\begin{lem} \label{lem:hole-free}
A fighting fish contains no holes.
\end{lem}

\begin{proof} Suppose a fighting fish $F$ contains a hole $C$, we take one of the rightmost cells, say $c$, of the cell cycle $C$. There exist cells $a,b\in C$ such that $a$ (resp. $b$) is attached to the left upper (resp. lower) edge of $c$. Note that there is no other cell attached to both $a$ and $b$. There are two possible ways to create the cell $c$ in $F$. In one way, the cell $c$ is obtained by a lower gluing on $a$ using \ref{enu:A2}. Then $c$ is not glued to $b$, a contradiction. In the other way, the cell $c$ is obtained by an upper gluing on $b$ using \ref{enu:A1}. Then $c$ is not glued to $a$, also a contradiction.
Thus, $F$ contains no holes.
\end{proof}

\begin{defi} \label{def:branch-cell} {\rm
A cell of a fighting fish is called a \emph{stem cell} if it satisfies one of the following conditions:
\begin{enumerate}
\item at least one of its two right edges is free;
\item it itself has no free right edges, but its right point is the intersection of two free edges. In this case, it is called a \emph{branch cell}.
\end{enumerate}
}
\end{defi}
Note that the bottom (resp. top) cell of each descending (resp. ascending) strip is a stem cell since it has a free right lower (resp. upper) edge. 
For example, the fighting fish shown in Figure \ref{fig:stem-cell} contains 9 stem cells, which are indicated by nodes. Among them, the two cells $w,a$ are branch cells, and the three cells $b,c,d$ are tails.

\begin{figure}[ht]
\begin{center}
\psfrag{a}[][][1]{$a$}
\psfrag{b}[][][1]{$b$}
\psfrag{c}[][][1]{$c$}
\psfrag{d}[][][1]{$d$}
\psfrag{e}[][][1]{$e$}
\psfrag{f}[][][1]{$f$}
\psfrag{g}[][][1]{$g$}
\psfrag{h}[][][1]{$h$}
\psfrag{w}[][][1]{$w$}
\includegraphics[width=1.6in]{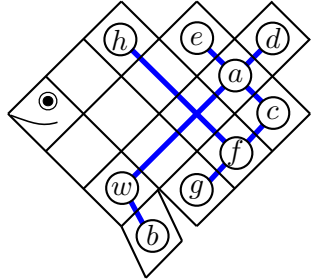}
\end{center}
\caption{\small A fighting fish and the associated graph on its stem cells.} \label{fig:stem-cell}
\end{figure}

\begin{lem} \label{lem:enclosure}
For any cell circuit $C$ of a fighting fish $F$,
if $u\in C$ is a cell such that the right point of $u$ is in the interior of a bounded region of the closed curve connecting the centers of consecutive cells of $C$ then $u$ is not a stem cell in $F$.
\end{lem}

\begin{proof} Consider the closed curve connecting the centers of consecutive cells along the cell circuit $C$. Suppose $u\in C$ is a cell such that the right point of $u$ is enclosed within a bounded region of this curve. By Lemma \ref{lem:hole-free}, the right upper (resp. lower) edge of $u$ is attached to a cell $a$ (resp. $b$) in $F$,  and there is a cell $d$ in $F$ attached to both $a$ and $b$.  It follows that neither the cell $u$ has a free right edge nor the right point of $u$ is the intersection of two free edges. Thus, $u$ is not a stem cell in $F$. The result follows.
\end{proof}

\begin{lem} \label{lem:number-stem-cells}
The following properties hold for every fighting fish $F\in\F_n$.
\begin{enumerate}
\item The number of branch cells in $F$ equals the number of tails minus 1.
\item There are $n$ stem cells in $F$.
\end{enumerate}
\end{lem}

\begin{proof} (i) Let $\tail(F)$ (resp. $\branch(F)$) be the number of tails (resp. branch cells) of $F$. We shall prove the assertion (i) by induction on the number of cells of $F$.

Initially, $F$ consists of a single cell, where $\tail(F)=1$ and $\branch(F)=0$. For $m>1$, suppose the assertion (i) holds for all fighting fish with $m-1$ cells. Let $F'$ be a fighting fish obtained by adding a new cell to a fighting fish $F$ with $m-1$ cells. Suppose the new cell is attached to a cell, say $t$, of $F$ using \ref{enu:A1} or \ref{enu:A2}. Then the cell $t$ has either one or two free right edges in $F$ and the new cell is a tail in $F'$. In the former case, the non-tail cell $t$ becomes a branch cell, and hence $\tail(F')=\tail(F)+1$ and $\branch(F')=\branch(F)+1$. In the latter case, the tail $t$ becomes a non-tail non-branch cell, and hence $\tail(F')=\tail(F)$ and $\branch(F')=\branch(F)$.

Suppose there is a cell $c$ in $F$ with two cells $a$ and $b$ attached to its right upper and lower edge and the new cell is attached to both $a$ and $b$ using \ref{enu:A3}. Then the new cell is a tail and the branch cell $c$ becomes a non-stem cell in $F'$. Moreover, the cell $a$ (resp. $b$) has either one or two free right edges in $F$. For the cell $a$ (resp. $b$), in the former case we turn it from a non-tail cell into a branch cell, and in the latter case from a tail into a non-tail non-branch cell. Hence $\tail(F')=\tail(F)$ and $\branch(F')=\branch(F)$. Thus, $\tail(F')-\branch(F')=\tail(F)-\branch(F)=1$. The result follows from induction.

(ii) Let $\stem(F)$ be the number of stem cells of $F$. Suppose there are $d_0$ (resp. $d_1$ and $d_2$) stem cells in $F$ having 0 (resp. 1 and 2) free right edges. Then $\stem(F)=d_0+d_1+d_2$ and $d_1+2d_2=n+1$. By (i), we have $d_0=d_2-1$. It follows that $\stem(F)=n$. 
\end{proof}

\begin{defi} \label{def:graph T_F} {\rm
For any fighting fish $F$, let $T_F$ be the graph defined as follows. 
\begin{enumerate}
\item The vertex set of $T_F$ is the set of stem cells of $F$.
\item Two vertices $a,b\in T_F$ are adjacent if the stem cells $a,b$ are in the same strip of $F$ and there is no other stem cell in the strip between $a$ and $b$.
\end{enumerate}
}
\end{defi}

The stem cell associated to a vertex $u$ in $T_F$ is also denoted by $u$. See Figure \ref{fig:stem-cell} for an example of the graph $T_F$ built on the stem cells of a fighting fish $F$. Note that the stem cells in the same strip of $F$ form a path in the graph $T_F$.

For any cell $c$ of a fighting fish, let $\Des(c)$ (resp. $\Asc(c)$) denote the descending (resp. ascending) strip containing the cell $c$. For convenience, the \emph{northeast} and \emph{southwest} directions of an ascending strip are denoted by the letters $N$ and $S$, respectively, and the \emph{southeast} and \emph{northwest} directions of a descending strip are denoted by the letters $E$ and $W$, respectively.

\begin{lem} \label{lem:TF-is-tree} For any fighting fish $F$, the graph $T_F$ is a tree.
\end{lem}

\begin{proof} We shall prove that the graph $T_F$ is acyclic and connected.
Suppose the graph $T_F$ contains a cycle, we take one of the leftmost stem cell, say $c$, of the cycle. There exist stem cells $a,b$ in this cycle adjacent to $c$ such that $a$ (resp. $b$) is on the $N$ (resp. $E$) side of $c$ in the strip $\Asc(c)$ (resp. $\Des(c)$) in $F$. Then there is another path $a=x_1, x_2,\dots,x_m=b$ of stem cells from $a$ to $b$ for some $m$ such that $x_i$ and $x_{i-1}$ are in the same strip for $2\le i\le m$. If we draw a closed curve connecting the centers of consecutive cells along the cycle $(c, x_1,\dots, x_m)$, we observe that the right point of $c$ will be enclosed within a bounded region of this curve. By Lemma \ref{lem:enclosure}, the cell $c$ is not a stem cell in $F$, a contradiction. Thus, $T_F$ is acyclic.

We shall prove the connectedness of $T_F$ by induction on the number of cells of $F$. The construction of $F$ starts from a cell and then adds new cells one by one using the operations \ref{enu:A1}--\ref{enu:A3}. Note that if $F$ contains at most three cells then it consists of stem cells only, and hence $T_F$ is connected. For $m>3$, suppose the assertion holds for all fighting fish with $m-1$ cells. Let $F'$ be a fighting fish obtained by adding a new cell to a fighting fish $F$ with $m-1$ cells. Suppose the new cell is attached to a cell, say $t$, of $F$ using \ref{enu:A1} or \ref{enu:A2}. Since the cell $t$ has either one or two free right edges in $F$, it becomes a branch cell or has a free right edge in $F'$. Thus, $t$ remains a stem cell in $F'$. The new cell, say $u$, is also a stem cell in $F'$. Since the node $u$ is adjacent to $t$, the resulting graph $T_{F'}$ remains connected.

Suppose there is a cell $c$ in $F$ with two cells $a$ and $b$ attached to its right upper and lower edge in $F$  and the new cell, say $d$, is attached to both $a$ and $b$ using \ref{enu:A3}. Then the branch cell $c$ becomes a non-stem cell and the cell $d$ is a tail in $F'$.
The resulting graph $T_{F'}$ is obtained from $T_F$ by replacing the vertex $c$ by $d$ and replacing the edge $ac$ (resp. $bc$) by $ad$ (resp. $bd$) (see Figure \ref{fig:connected}(i)). Moreover, if $c$ is connected to a stem cell $w$ (resp. $w'$) in the strip $\Asc(c)$ (resp. $\Des(c)$) or connected to both $w$ and $w'$, then the edge $wc$ (resp. $w'c$) is replaced by the edge $wa$ (resp. $w'b$) (see Figure \ref{fig:connected}(ii)). Thus, $T_F'$ remains connected. The result follows from induction.
\end{proof}

\begin{figure}[ht]
\begin{center}
\psfrag{a}[][][1]{$a$}
\psfrag{b}[][][1]{$b$}
\psfrag{c}[][][1]{$c$}
\psfrag{d}[][][1]{$d$}
\psfrag{w}[][][1]{$w$}
\psfrag{w'}[][][1]{$w'$}
\includegraphics[width=4.2in]{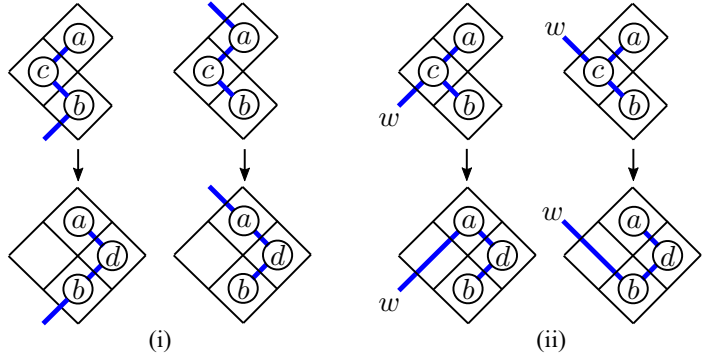}
\end{center}
\caption{\small The varied stem cells by an operation \ref{enu:A3}.} \label{fig:connected}
\end{figure}

\section{On fighting fish with a marked stem cell} \label{sec:marked-cell}

In this section, we shall establish the bijection $\varphi:\QT_n\rightarrow\CF_n$ in Theorem \ref{thm:fighting-fish-marked-stem-cell}.

\subsection{A vertex-labeling representation of ternary trees}
We introduce a representation of ternary trees in terms of a vertex-labeling on rooted trees.

\begin{defi} \label{def:NESW} {\rm
Let $X$ be an arbitrary letter in $\{N,E,S,W\}$.
For any ternary tree $T\in\T_n$, let $\chi_X(T)$ denote the underlying rooted tree of $T$ with a vertex-labeling $\lambda: T \rightarrow \{N, E, S, W\}$ defined by the following rules \ref{enu:B1}--\ref{enu:B5}.
For any internal node $w$ of $T$, let $x,y$ and $z$ be the left, middle and right child of $w$ (if any). 

\setcounter{algo}{2}
\begin{enumerate}[label=(\thealgo\arabic*), ref=(\thealgo\arabic*)]
\item If $w$ is the root,  set $\lambda(w)=X$. \label{enu:B1}
\item If $\lambda(w)=E$,  set $\lambda(x)=N$, $\lambda(y)=E$ and $\lambda(z)=S$. \label{enu:B2}
\item If  $\lambda(w)=N$,  set $\lambda(x)=E$, $\lambda(y)=N$ and $\lambda(z)=W$. \label{enu:B3}
\item If $\lambda(w)=W$,  set $\lambda(x)=N$, $\lambda(y)=W$ and $\lambda(z)=S$. \label{enu:B4}
\item If $\lambda(w)=S$,  set $\lambda(x)=E$, $\lambda(y)=S$ and $\lambda(z)=W$. \label{enu:B5}
\end{enumerate}
For each $X\in\{N,E,S,W\}$, let $\T^X_n:=\{\chi_X(T)\mid T\in\T_n\}$.
}
\end{defi}

For example,  let $X=E$ and take the ternary tree $T$ shown on the left of Figure \ref{fig:ternary-tree}. The corresponding tree $\chi_E(T)$ is shown on the right.
\begin{figure}[ht]
\begin{center}
\psfrag{a}[][][1]{$a$}
\psfrag{b}[][][1]{$b$}
\psfrag{c}[][][1]{$c$}
\psfrag{d}[][][1]{$d$}
\psfrag{e}[][][1]{$e$}
\psfrag{f}[][][1]{$f$}
\psfrag{g}[][][1]{$g$}
\psfrag{h}[][][1]{$h$}
\psfrag{j}[][][1]{$j$}
\psfrag{k}[][][1]{$k$}
\psfrag{w}[][][1]{$w$}
\includegraphics[width=2.8in]{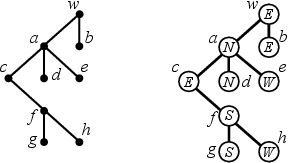}
\end{center}
\caption{\small A ternary tree in $\T_{9}$ and its vertex-labeling representation in $\T^E_{9}$.} \label{fig:ternary-tree}
\end{figure}

\begin{lem} \label{lem:X=NESW}
For all $n\ge 0$, the map $\chi_X:\T_n\rightarrow\T^X_n$ defined in Definition \ref{def:NESW} is a bijection for each $X\in\{N,E,S,W\}$. 
\end{lem}

\begin{proof} Note that the map $\chi_X$ is trivial for $n\le 1$.
For $n\ge 2$ and any tree $T\in\T^X_n$, the corresponding ternary tree ${\chi_X}^{-1}(T)$ is obtained from $T$ by imposing an ordering on sibling nodes defined as follows.
For any internal node $w\in T$ and a child $u$ of $w$, the position of $u$ is determined according to the following four cases.

\setcounter{algo}{3}
\begin{enumerate}[label=(\thealgo\arabic*), ref=(\thealgo\arabic*)]
\item $\lambda(w)=E$. Set $u$ to be the left (resp. middle and right) child of $w$ if $\lambda(u)=N$ (resp. $E$ and $S$). \label{enu:C1}
\item $\lambda(w)=N$. Set $u$ to be the left (resp. middle and right) child of $w$ if $\lambda(u)=E$ (resp. $N$ and $W$). \label{enu:C2}
\item $\lambda(w)=W$. Set $u$ to be the left (resp. middle and right) child of $w$ if $\lambda(u)=N$ (resp. $W$ and $S$). \label{enu:C3}
\item $\lambda(w)=S$. Set $u$ to be the left (resp. middle and right) child of $w$ if $\lambda(u)=E$ (resp. $S$ and $W$). \label{enu:C4}
\end{enumerate}
Note that \ref{enu:C1}--\ref{enu:C4} are exactly the reverse operations of \ref{enu:B2}--\ref{enu:B5}. The result follows.
\end{proof}

\begin{defi} \label{def:underlying-tree} {\rm
For each tree $(t;T_1,T_2,T_3,T_4)\in\QT_n$, we associate it with a labeled rooted tree with $n$ nodes such that the root $t$ has the four subtrees $\chi_N(T_1),\chi_E(T_2),\chi_S(T_3)$ and $\chi_W(T_4)$. The root $t$ is labeled by an arbitrary letter in $\{N,E,S,W\}$.
}
\end{defi}

The following result characterizes the labeled rooted trees associated to the trees in $\QT_n$.

\begin{lem} \label{lem:NESW-representation}
A rooted tree $T$ with $n$ nodes and with a vertex-labeling $\lambda:T\rightarrow\{N,E,S,W\}$ is associated to a tree in $\QT_n$ if and only if $T$ satisfies the following conditions.
\begin{enumerate}
\item The root of $T$ has at most four children, and the children of the root have distinct labels.
\item For any non-root node $u\in T$, the node $u$ has at most three children, and the children of $u$ have distinct labels, which are all different from the opposite directional letter of $\lambda(u)$.
\end{enumerate}
\end{lem}

\begin{proof} Let $T$ be a tree associated to a tree in $\QT_n$. By Definition \ref{def:underlying-tree}, $T$ satisfies condition (i). By the rules \ref{enu:B1}--\ref{enu:B5},  $T$ satisfies condition (ii).

On the other hand, let $T$ be a rooted tree with $n$ nodes and with a vertex-labeling $\lambda: T\rightarrow\{N,E,S,W\}$ satisfying conditions (i) and (ii). By condition (i), let $w$, $x$, $y$ and $z$ be the children of the root $t$ with label $N$, $E$, $S$ and $W$, respectively (if any). Let $T_w$ (resp. $T_x$, $T_y$ and $T_z$) be the subtree of $T$ rooted at $w$ (resp. $x$, $y$ and $z$). By conditions (ii), we observe that each of $T_w$, $T_x$, $T_y$ and $T_z$ can be converted into a ternary tree using \ref{enu:C1}--\ref{enu:C4} in the proof of Lemma \ref{lem:X=NESW}. Thus, $T$ is associated to the ordered tree $(t;T_1,T_2,T_3,T_4)\in\QT_n$, where $T_1={\chi_N}^{-1}(T_w)$, $T_2={\chi_E}^{-1}(T_x)$, $T_3={\chi_S}^{-1}(T_y)$ and $T_4={\chi_W}^{-1}(T_z)$. The result follows.
\end{proof}

\subsection{A map $\varphi:\QT_n\rightarrow\CF_n$}

\begin{defi} {\rm
Let $T$ be a rooted tree with $n$ nodes. By a \emph{chain} $T_1\subset T_2\subset\cdots\subset T_n$ \emph{of subtrees} of $T$, we mean $T_n=T$ and $T_i$ is obtained from $T_{i+1}$ by removing a leaf, for all $i<n$.
}
\end{defi}

We describe a map $\varphi:\QT_n\rightarrow\CF_n$. Given a tree in $\QT_n$, let $T$ be the associated labeled rooted tree. We shall construct a fighting fish from $T$ using the following Algorithm D. We remark that the label of the root is irrelevant for the construction.

\begin{algo} \label{algo:tree-to-fish}
\leavevmode

Choose a chain $T_1\subset T_2\subset\cdots\subset T_n$ of subtrees of $T$. Starting from $T_1$, we construct a union of cells from  $T_i$, denoted by $U_i$,  for $i=1,2,\dots,n$ as follows.

Let $U_1$ be a single cell with the label of the root. For $i>1$, suppose $U_{i-1}$ has been constructed. Let $u$ be the node in $T_i\setminus T_{i-1}$. We shall construct $U_i$ from $U_{i-1}$ by adding cells, so that the cell associated to $u$ is created in $U_i$. Let $w$ be the parent of $u$ in $T$. (We shall see in Lemma \ref{lem:cell-union}(i) that the cell $w$ is a stem cell in the fighting fish $U_{i-1}$). There are four possible ways to add cells according to $\lambda(u)$.

\begin{enumerate}[label=(\thealgo\arabic*), ref=(\thealgo\arabic*)]
\item $\lambda(u)=N$. Then the right upper edge of $w$ is free in $U_{i-1}$ (see Lemma \ref{lem:cell-union}(i)). We attach the new cell $u$ to the right upper edge of $w$. \label{enu:D1}
\item $\lambda(u)=E$.  Then the right lower edge of $w$ is free (see Lemma \ref{lem:cell-union}(i)). We attach the new cell $u$ to the right lower edge of $w$. \label{enu:D2}
\item $\lambda(u)=W$.  Then there is no stem cell on the $W$ side of $w$ in $U_{i-1}$ (see Lemma \ref{lem:cell-union}(i)).
Let $q$ be the top cell of the strip $\Des(w)$, and let $p$ be the bottom cell of the strip $\Asc(q)$ (see Figure \ref{fig:upper-insertion}(i)). 
We insert an ascending strip $R$ on the immediate left of $\Asc(q)$ by duplicating the cells from $p$ to $q$. (see Figure \ref{fig:upper-insertion}(ii)). In this case, the top cell of $R$ is the new cell associated to the node $u$. \label{enu:D3}
\item $\lambda(u)=S$.  Then there is no stem cell on the $S$ side of $w$ in $U_{i-1}$ (see Lemma \ref{lem:cell-union}(i)). Let $q$ be the bottom cell of $\Asc(w)$, and $p$ be the top cell of $\Des(q)$ (see Figure \ref{fig:lower-insertion}(i)). \label{enu:D4}
We insert a descending strip $R$ on the immediate left of $\Des(q)$ by duplicating the cells from $p$ to $q$. (see Figure \ref{fig:lower-insertion}(ii)). The bottom cell of $R$ is the new cell associated to $u$.
\end{enumerate}
In \ref{enu:D1}--\ref{enu:D4}, the cell $u$ is added to the $\lambda(u)$ side of $w$. So, the cell $u$ carries the label $\lambda(u)$.
\end{algo}

\begin{figure}[ht]
\begin{center}
\psfrag{w}[][][1]{$w$}
\psfrag{q}[][][1]{$q$}
\psfrag{p}[][][1]{$p$}
\psfrag{b}[][][1]{$b_1$}
\psfrag{c}[][][1]{$c_1$}
\psfrag{d}[][][1]{$b_k$}
\psfrag{f}[][][1]{$c_k$}

\includegraphics[width=5.0in]{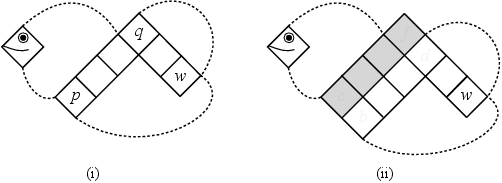}
\end{center}
\caption{\small The operation \ref{enu:D3} for adding an ascending strip of cells $c_1,\dots,c_k$.} \label{fig:upper-insertion}
\end{figure}

\begin{figure}[ht]
\begin{center}
\psfrag{w}[][][1]{$w$}
\psfrag{q}[][][1]{$q$}
\psfrag{p}[][][1]{$p$}
\psfrag{b}[][][1]{$b_1$}
\psfrag{c}[][][1]{$c_1$}
\psfrag{d}[][][1]{$b_k$}
\psfrag{f}[][][1]{$c_k$}

\includegraphics[width=5.0in]{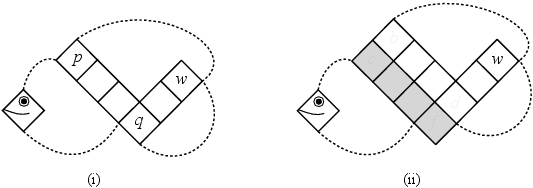}
\end{center}
\caption{\small The operation \ref{enu:D4} for adding a descending strip of cells $c_1,\dots,c_k$.} \label{fig:lower-insertion}
\end{figure}

\begin{exa} \label{exa:stem-cell-construction} {\rm
Consider the labeled rooted tree $T$ shown in Figure \ref{fig:ternary-tree}. Using Algorithm \ref{algo:tree-to-fish}, we show in Figure \ref{fig:fish-construction} a chain of subtrees of $T$ and the union of cells constructed from each subtree.
}
\end{exa}

\begin{figure}[ht]
\begin{center}
\psfrag{a}[][][1]{$a$}
\psfrag{b}[][][1]{$b$}
\psfrag{c}[][][1]{$c$}
\psfrag{d}[][][1]{$d$}
\psfrag{e}[][][1]{$e$}
\psfrag{f}[][][1]{$f$}
\psfrag{g}[][][1]{$g$}
\psfrag{h}[][][1]{$h$}
\includegraphics[width=5.4in]{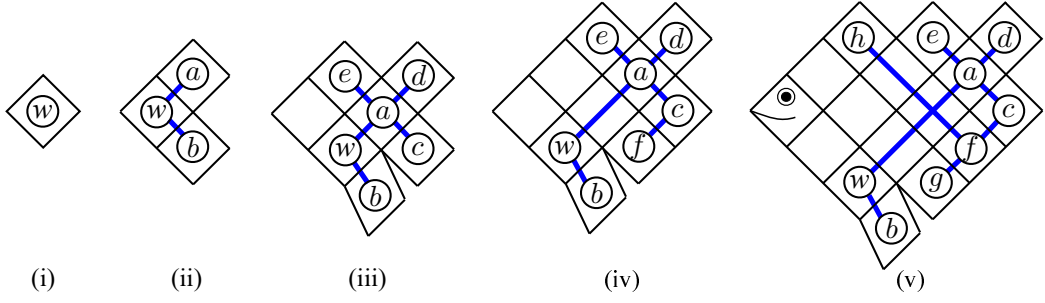}
\end{center}
\caption{\small The construction of the fighting fish from the ternary tree in Figure \ref{fig:ternary-tree}.} \label{fig:fish-construction}
\end{figure}

\begin{lem} \label{lem:union-cell-fish}
Given a tree in $\QT_n$, let $T$ be the associated labeled rooted tree. For any chain $T_1\subset \cdots\subset T_n$ of subtrees of $T$, let $U_j$ be the union of cells constructed from $T_j$ using Algorithm \ref{algo:tree-to-fish}, for $1\le j\le n$. Then each $U_j$ is a fighting fish.
\end{lem}

\begin{proof} We shall prove the assertion by induction. For $j=1$, $U_1$ is a fighting fish with a single cell. For $j>1$, suppose $U_{j-1}$ is a fighting fish. To prove that $U_j$ is a fighting fish, it suffices to prove that $U_j$ can be constructed properly using \ref{enu:A1}--\ref{enu:A3}. Let $u=T_j\setminus T_{j-1}$ and let $w$ be the parent of $u$.

If $\lambda(u)=N$ (resp. $E$) then by \ref{enu:D1}--\ref{enu:D2} $U_j$ is obtained from $U_{j-1}$ by an upper (resp. lower) gluing of $u$ on $w$. Thus, $U_j$ can be constructed using \ref{enu:A1}--\ref{enu:A3}. 

Suppose $\lambda(u)=W$. In $U_{j-1}$, let $q$ be the top cell of the strip $\Des(w)$, and let $p$ be the bottom cell of $\Asc(q)$ (see Figure \ref{fig:upper-insertion}(i)). For some integer $k$, let $p=b_1, b_2,\dots, b_k=q$ be the consecutive cells from $p$ to $q$ in the strip $\Asc(q)$. By \ref{enu:D3}, $U_j$ is obtained from $U_{j-1}$ by inserting an ascending strip of $k$ cells on the immediate left of $\Asc(q)$, which are duplicates of the cells from $p$ to $q$. This operation can be accomplished using \ref{enu:A1}--\ref{enu:A3} as follows.

During the construction of $U_{j-1}$, since $b_1$ is the bottom cell of $\Asc(q)$, the cell $b_1$ either is the head or is created by a lower gluing \ref{enu:A2}. In the former case, $U_j$ is obtained by setting a cell $c_1$ to be the head and perform a lower gluing of $b_1$ on $c_1$. In the latter case, instead of gluing $b_1$, we glue first $c_1$ and then $b_1$. For $i\ge 2$, instead of gluing $b_i$ on $b_{i-1}$, we glue first $c_i$ on $c_{i-1}$ and then perform a double gluing of $b_i$ on $c_i$ and $b_{i-1}$ (see Figure \ref{fig:upper-insertion}(ii)).
Thus, $U_j$ can be constructed using \ref{enu:A1}--\ref{enu:A3}. A symmetric argument applies to the case $\lambda(u)=S$. The assertion follows.
\end{proof}

\begin{lem} \label{lem:cell-union}
Given a tree in $\QT_n$, let $T$ be the associated labeled rooted tree. For any chain $T_1\subset \cdots\subset T_n$ of subtrees of $T$, let $U_j$ be the fighting fish constructed from $T_j$ using Algorithm \ref{algo:tree-to-fish}, for $1\le j\le n$. Then $U_1$ consists of a single cell and the following results hold.
\begin{enumerate}
\item For each $j>1$, let $u$ denote the node in $T_j\setminus T_{j-1}$, and let $w$ denote the parent of $u$. Then the cell $w$ is a stem cell in $U_{j-1}$. If $\lambda(u)=N$ (resp. $E$) then the right upper (resp. lower) edge of $w$ is free in $U_{j-1}$. If $\lambda(u)=W$ (resp. $S$) then there is no stem cell on the $W$ (resp. $S$) side of $w$ in $U_{j-1}$.
\item For each $j>1$, the stem cells of $U_j$ consist of the cells associated to the nodes of $T_j$ and carry the same labeling as in $T_j$.
\end{enumerate}
\end{lem}

\begin{proof} (i) By Algorithm \ref{algo:tree-to-fish}, $U_1$ is a single cell, denoted by $t$, with the label of the root. We shall prove the assertions (i) and (ii) by induction. For the initial case $j=2$, let $u=T_2\setminus T_1$.  Note that the cell $t$ is a stem cell in $U_1$ (with four free edges). If $\lambda(u)=N$ (resp. $E$) then by \ref{enu:D1}--\ref{enu:D2} the cell $u$ is attached to the right upper (resp. lower) edge of $t$. If $\lambda(u)=W$ (resp. $S$) then by \ref{enu:D3}--\ref{enu:D4} the cell $u$ is attached to the left upper (resp. lower) edge of $t$. Thus, $U_2$ consists of two stem cells with the same labeling as in $T_2$.  For $j>2$, suppose the assertions (i) and (ii) hold for the case $j-1$. Let $u=T_j\setminus T_{j-1}$ and let $w$ be the parent of $u$. By induction hypothesis, the cell $w$ is a stem cell in $U_{j-1}$. 

Suppose $\lambda(u)=N$. Suppose the right upper edge of $w$ is not free in $U_{j-1}$, there is a stem cell, say $q$, connected to $w$ on the $N$ side of $w$. If $q$ is another child of $w$ then the cell $q$ is created by an operation \ref{enu:D1} on $w$ during the construction of $U_{j-1}$, and hence $\lambda(q)=N$. By Lemma \ref{lem:NESW-representation}, this contradicts the condition $\lambda(q)\neq\lambda(u)$. Otherwise, $w$ is a child of $q$, and hence $w$ is a non-root node. Since the cell $w$ is on the $S$ side of $q$, it is created by an operation \ref{enu:D4} on $q$, and hence $\lambda(w)=S$. By Lemma \ref{lem:NESW-representation}(ii),  this contradicts the condition that $\lambda(u)$ is different from the opposite of $\lambda(w)$. Thus, the right upper edge of $w$ is free in $U_{j-1}$. 

By \ref{enu:D1}, $U_j$ is obtained from $U_{j-1}$ by attaching $u$ to the right upper edge of $w$. Note that $w$ is the only stem cell in $U_{j-1}$ affected by this operation. If $w$ is a tail of $U_{j-1}$ then the right lower edge of $w$ is free in $U_j$; otherwise, $w$ becomes a branch cell. Thus, $w$ is still a stem cell in $U_j$, and hence all of the stem cells of $U_{j-1}$ remain stem cells in $U_j$. A symmetric argument applies to the case $\lambda(u)=E$. 

The assertion that if $\lambda(u)=W$ (resp. $S$) then there is no stem cell on the $W$ (resp. $S$) side of $w$ in $U_{j-1}$ can also be proved by a similar argument. By \ref{enu:D3} (resp. \ref{enu:D4}), $U_j$ is obtained from $U_{j-1}$ by inserting an ascending (resp. descending) strip, so that the cell $u$ is created on the $W$ (resp. $S$) side of $w$. Note that the free right edges of $U_{j-1}$ remain unchanged, and hence all of the stem cells of $U_{j-1}$ remain stem cells in $U_j$. The assertions (i) and (ii) follow from induction.
\end{proof}

\begin{lem} \label{lem:leaf-cells} 
Given a tree in $\QT_n$, let $T$ be the associated labeled rooted tree. For any chain $T_1\subset\cdots\subset T_n$ of subtrees of $T$, let $F_j$ be the fighting fish constructed from $T_j$ using Algorithm \ref{algo:tree-to-fish}. 
For each $T_j$ and $j>1$, if $u$ is a leaf in $T_j$ then the following properties of $F_j$ hold.
\begin{enumerate}
\item If $\lambda(u)=N$ then the cell $u$ is the top cell of the strip $\Asc(u)$, and the strip $\Des(u)$ consists of the cell $u$ itself in $F_j$.
\item If $\lambda(u)=E$ then the cell $u$ is the bottom cell of the strip $\Des(u)$, and the strip $\Asc(u)$ consists of the cell $u$ itself in $F_j$.
\item If $\lambda(u)=W$ then the cell $u$ is the top cell of both $\Des(u)$ and $\Asc(u)$, and it is also the unique stem cell of $\Asc(u)$ in $F_j$. 
\item If $\lambda(u)=S$ then the cell $u$ is the bottom cell of both $\Asc(u)$ and $\Des(u)$, and it is also the unique stem cell of $\Des(u)$ in $F_j$.
\end{enumerate}
\end{lem}
 
\begin{proof} For an integer $j>1$, suppose $u$ is a leaf in $T_j$, let $w$ be the parent of $u$. Consider the label of $u$.

1.
Suppose $\lambda(u)=N$. By \ref{enu:D1}, the cell $u$ is attached to the right upper edge of $w$ during the construction of $F_j$. Since $u$ is a leaf of $T_j$, the cell $u$ is the unique stem cell on the $N$ side of $w$ in   $F_j$, and hence $u$ is the top cell of $\Asc(u)$. We claim that the strip $\Des(u)$ consists of the cell $u$ itself in $F_j$.

Suppose there are other cells on the $E$ side of $u$ in  $F_j$. Since the bottom cell of $\Des(u)$ is a stem cell, there is at least one stem cell on the $E$ side of $u$. Then the cell $u$ is connected to another stem cell of $\Des(u)$ in $F_j$, which is against the condition that $u$ is a leaf in $T_j$.

Suppose there are other cells on the $W$ side of $u$ in $F_j$, let $v$ be the cell attached to the left upper edge of $u$. We observe that either the right upper edge of $v$ is free, or the right point of $v$ is the intersection of two free edges. Thus, $v$ is a stem cell connected to $u$ in $F_j$, which is against the same condition. 
Thus, the strip $\Des(u)$ consists of the cell $u$ itself in $F_j$.  A symmetric argument applies to the case $\lambda(u)=E$. The assertions (i) and (ii) follow.

2.
Suppose $\lambda(u)=W$. By \ref{enu:D3}, the cell $u$ is created on the $W$ side of $w$. Since $u$ is a leaf in $T_j$, the cell $u$ is the unique stem cell on the $W$ side of $w$ in the strip $\Des(w)$ of $F_j$, and hence $u$ is the top cell of $\Des(u)$. We claim that $u$ is the top cell of $\Asc(u)$ and that it is also the unique stem cell of $\Asc(u)$ in $F_j$.

Suppose there are other cells on the $N$ side of $u$ in $F_j$. Since the top cell of $\Asc(u)$ is a stem cell, there is at least one stem cell on the $N$ side of $u$. Then the cell $u$ is connected to another stem cell of $\Asc(u)$ in $F_j$, a contradiction. Thus, $u$ is the top cell of $\Asc(u)$ in $F_j$. Note that there are possibly other cells on the $S$ side of $u$, but if there exists any stem cell among them then the cell $u$ is connected to anther stem cell of $\Asc(u)$ in $F_j$, also a contradiction. Thus, $u$ is the unique stem cell of $\Asc(u)$ in $F_j$. A symmetric argument applies to the case $\lambda(u)=S$. The assertions (iii) and (iv) follow.
\end{proof}

The following result shows that during the construction of a fighting fish the exchange of two successive but independent operations does not change the outcome.

\begin{lem} \label{lem:exchange-operation} 
Given a tree in $\QT_n$ with at least two leaves, let $T$ be the associated labeled rooted tree.  Let $\tau:T_1\subset \cdots\subset T_n$ be a chain of subtrees of $T$. For some $i>1$, let $u=T_i\setminus T_{i-1}$ and $v = T_{i+1}\setminus T_i$ such that $v$ is not a child of $u$. Let $\tau'$ be the chain of subtrees of $T$ obtained from $\tau$ by replacing $T_i$ by $(T_i\setminus\{u\})\cup\{v\}$. Let $F, F'$ be the fighting fish constructed from $\tau,\tau'$, respectively using Algorithm \ref{algo:tree-to-fish}. Then $F=F'$.
\end{lem}

\begin{proof}
For the chain $\tau:T_1\subset\cdots\subset T_n$, let $F_j$ be the fighting fish constructed from $T_j$. For the chain $\tau'$, let $\tau'$ be denoted as $\tau':T'_1\subset\cdots\subset T'_n$, where
\begin{equation}
T'_j = \begin{cases}
(T_j\setminus\{u\})\cup\{v\} & \mbox{if $j=i$;} \\
T_j & \mbox{otherwise,}
\end{cases}
\end{equation}
and let $F'_j$ be the fighting fish constructed from $T'_j$. Since $T'_j=T_j$ for all $j\le i-1$, we have $F'_{i-1}=F_{i-1}$.  We claim that $F'_{i+1}=F_{i+1}$.

Since $v$ is not a child of $u$, the node $u$ is a leaf in $T_{i+1}$. Let $w$ be the parent of $u$. Consider the label of $u$.

1. Suppose $\lambda(u)=N$. By Lemma \ref{lem:leaf-cells}(i), the cell $u$ is the top cell of $\Asc(u)$, and the strip $\Des(u)$ consists of the cell $u$ itself in $F_{i+1}$. Note that $T'_i=T_{i+1}\setminus\{u\}$. It follows that $F'_i$ is exactly the fighting fish obtained from $F_{i+1}$ by removing the cell $u$. Since $T'_{i+1}=T_{i+1}=T'_i\cup\{u\}$, by \ref{enu:D1} $F'_{i+1}$ is obtained from $F'_i$ by attaching the cell $u$ to the right upper edge of $w$. Thus, $F'_{i+1}=F_{i+1}$. A symmetric argument applies to the case $\lambda(u)=E$.

2. Suppose $\lambda(u)=W$. By Lemma \ref{lem:leaf-cells}(iii), the cell $u$ is the top cell of both $\Des(u)$ and $\Asc(u)$, and it is also the unique stem cell of $\Asc(u)$ in $F_{i+1}$. Since $T'_i=T_{i+1}\setminus\{u\}$, $F'_i$ is exactly the fighting fish obtained from $F_{i+1}$ by removing the strip $\Asc(u)$. Since $T'_{i+1}=T'_i\cup\{u\}$, by \ref{enu:D3} $F'_{i+1}$ is obtained from $F'_i$ by inserting an ascending strip, so that the cell $u$ is added to the $W$ side of $w$. Thus, $F'_{i+1}=F_{i+1}$. A symmetric argument applies to the case $\lambda(u)=S$.

Since $T'_j=T_j$ for all $j\ge i+1$, we have $F'_n=F_n$. The result follows.
\end{proof}

The following result shows that a ternary tree determines a unique fighting fish, independent of the choice of chain of subtrees.

\begin{pro} \label{pro:isomorphic}
Given a tree in $\QT_n$, let $T$ be the associated labeled rooted tree. For any two chains $\tau:T_1\subset\cdots\subset T_n$ and $\tau':T'_1\subset\cdots\subset T'_n$ of subtrees of $T$, let $F, F'\in\F_n$ be the fighting fish constructed from $\tau,\tau'$, respectively using Algorithm $D$. Then $F=F'$.
\end{pro}

\begin{proof} If $T$ has only one leaf then we have $\tau=\tau'$ and hence $F=F'$. Suppose $T$ has at lease two leaves.
We claim that there is a sequence of chains $\tau'=\tau_1,\tau_2,\dots,\tau_m=\tau$ of subtrees of $T$ for some $m$ such that $\tau_{i+1}$ is obtained from $\tau_{i}$ by a serial exchanges of successive but independent operations, so that the constructions from $\tau_{i+1}$ and $\tau_i$ have the same outcome.

We prove the claim by induction on the least index $j$ such that $T'_j\neq T_j$.
Note that $T'_1=T_1$, consisting of the root of $T$. Let $\ell$ be the least integer such that $T'_{\ell}\neq T_{\ell}$ for some $\ell>1$. 
Let $v=T_{\ell}\setminus T_{\ell-1}$ and $q=T'_{\ell}\setminus T'_{\ell-1}$. 
Suppose $v=T'_k\setminus T'_{k-1}$ for some $k>\ell$, let $q=u_0, u_1,\dots, u_{k-\ell}=v$ be the nodes such that $u_j=T'_{\ell+j}\setminus T'_{\ell+j-1}$ for $0\le j\le k-\ell$. Note that $v$ is not a child of each $u_j$ for $j< k-\ell$ since the parent of $v$ is in $T_{\ell-1}=T'_{\ell-1}$.  We construct a chain $\tau'':T''_1\subset\cdots\subset T''_n$ from $\tau'$ by setting 
\begin{equation}
T''_j =\begin{cases}
(T'_j\setminus\{u_{j-\ell}\})\cup\{v\} & \mbox{if $\ell\le j\le k-1$;} \\
T'_j & \mbox{otherwise.}
\end{cases}
\end{equation}
That is, $T''_{k-1}=T'_k\setminus\{u_{k-1-\ell}\}$, which is obtained by exchanging $v$ for $u_{k-1-\ell}$ in $T'_{k-1}$, and then $T''_j=T''_{j+1}\setminus\{u_{j-\ell}\}$, which is obtained by exchanging  $v$ for $u_{j-\ell}$ in $T'_{j}$ for $\ell\le j\le k-2$. 
By Lemma \ref{lem:exchange-operation}, the constructions from $\tau',\tau''$ have the same outcome. Note that $T''_{\ell}=(T'_{\ell}\setminus\{q\})\cup\{v\}=T_{\ell}$ and $T''_j=T_j$ for all $j< \ell$.
By induction, the announced sequence $\tau'=\tau_1,\tau_2,\dots,\tau_m=\tau$ of chains can be obtained. The result follows.
\end{proof}

\subsection{A map $\varphi':\CF_n\rightarrow\QT_n$}
We describe a map $\varphi':\CF_n\rightarrow\QT_n$. Given a pair $(F, t)\in\CF_n$, let $T_{(F,t)}$ denote the labeled rooted tree on the stem cells of $F$ obtained by specifying the cell $t$ to be the root and a vertex-labeling $\lambda:T_{(F,t)}\rightarrow\{N,E,S,W\}$ as follows. 

\begin{algo} \label{algo:fish-to-tree}
\leavevmode

Let $X$ be an arbitrary letter in $\{N,E,S,W\}$.
\begin{enumerate}[label=(\thealgo\arabic*), ref=(\thealgo\arabic*)]
\item Set $\lambda(t)=X$. \label{enu:E1}
\item For any labeled node $u$, if $v$ is an unlabeled node adjacent to $u$ on the $N$ (resp. $E$, $S$ and $W$) side of $u$, we set $v$ to be a child of $u$ with $\lambda(v)=N$ (resp. $E$, $S$ and $W$). \label{enu:E2}
\end{enumerate}
In this case, the cell $u$ is called the \emph{parent cell} of $v$, and the cell $t$ is called the \emph{root cell}.
\end{algo}

\begin{exa} {\rm
Consider the fighting fish $F$ shown in Figure \ref{fig:fish-to-tree}(i), along with the tree on its stem cells. If the cell $w$ is set to be the root with the choice $\lambda(w)=E$, the labeled rooted tree $T_{(F,w)}$ is shown in Figure \ref{fig:fish-to-tree}(ii). (The ternary tree formation of $T_{(F,w)}$ is shown in Figure \ref{fig:ternary-tree}.)
}
\end{exa}

\begin{figure}[ht]
\begin{center}

\psfrag{a}[][][1]{$a$}
\psfrag{b}[][][1]{$b$}
\psfrag{c}[][][1]{$c$}
\psfrag{d}[][][1]{$d$}
\psfrag{e}[][][1]{$e$}
\psfrag{f}[][][1]{$f$}
\psfrag{g}[][][1]{$g$}
\psfrag{h}[][][1]{$h$}
\psfrag{w}[][][1]{$w$}
\includegraphics[width=3.8in]{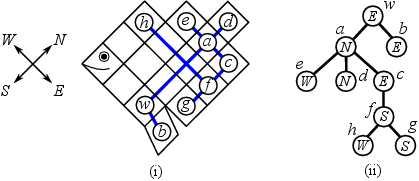}
\end{center}
\caption{\small A fighting fish and the associated labeled rooted tree.} \label{fig:fish-to-tree}
\end{figure}

\begin{lem} \label{lem:NESW-fish-tree}
For any pair $(F, t)\in\CF_n$, let $T_{(F,t)}$ be the labeled rooted tree on the stem cells of $F$ obtained using Algorithm \ref{algo:fish-to-tree}. Then the following properties hold.
\begin{enumerate}
\item The root of $T_{(F,t)}$ has at most four children, and the children of the root have distinct labels.
\item For any non-root node $u\in T_{(F,t)}$, the node $u$ has at most three children, and the children of $u$ have distinct labels, which are all different from the opposite directional letter of $\lambda(u)$.
\end{enumerate}
\end{lem}

\begin{proof} If $p,q$ are two children of the root $t$ of $T_{(F,t)}$ then by \ref{enu:E2} the cell $p$ (resp. $q$) is on the $\lambda(p)$ (resp. $\lambda(q)$) side of the cell $t$ in $F$, where $\lambda(p),\lambda(q)\in\{N,E,S,W\}$ and $\lambda(p)\neq\lambda(q)$. The assertion (i) follows.

For any non-root node $u\in T_{(F,t)}$, let $w$ be the parent of $u$. Suppose $\lambda(u)=N$, by \ref{enu:E2} the  cell $w$ is on the $S$ side of the cell $u$ in $F$. If $p,q$ are two children of $u$ in $T_{(F,t)}$ then by \ref{enu:E2} the cell $p$ (resp. $q$) is on the $\lambda(p)$ (resp. $\lambda(q)$) side of $u$ in $F$, where $\lambda(p),\lambda(q)\in\{N,E,W\}$ and $\lambda(p)\neq\lambda(q)$. The same argument applies to the cases $\lambda(u)\in\{E,W,S\}$. The assertion (ii) follows.
\end{proof}

\smallskip
Now, we describe the bijection $\varphi:\QT_n\rightarrow\CF_n$ in Theorem \ref{thm:fighting-fish-marked-stem-cell}.

\medskip 
\noindent
\emph{Proof of Theorem \ref{thm:fighting-fish-marked-stem-cell}.} 
Given a tree $(t;T_1,T_2,T_3,T_4)\in\QT_n$, let $T$ be the associated labeled rooted tree, where the root $t$ has the four subtrees ${\chi_N}(T_1)$, ${\chi_E}(T_2)$, ${\chi_S}(T_3)$ and ${\chi_W}(T_4)$ (possibly empty). Choose a chain $\tau$ of subtrees of $T$, and let $F$ be the fighting fish with the root cell $t$ constructed from the chain $\tau$ using Algorithm D. By Proposition \ref{pro:isomorphic}, the fighting fish $F$ is uniquely determined by $T$ independent of the choice of $\tau$. We define the map $\varphi$ by setting $\varphi:(t;T_1,T_2,T_3,T_4)\mapsto (F,t)\in\CF_n$.

For the tree $(t;T_1,T_2,T_3,T_4)$, let $w$ (resp. $x$, $y$ and $z$) be the root of $T_1$ (resp. $T_2$, $T_3$ and $T_4$). Then the stem cell $w$ (resp. $x$, $y$ and $z$) is on the $N$ (resp. $E$, $S$ and $W$) side of the root cell $t$ in $F$. Consider a labeled rooted tree $T'$ where the root $t$ has the four subtrees  ${\chi_E}(T_1)$, ${\chi_N}(T_2)$, ${\chi_W}(T_3)$ and ${\chi_S}(T_4)$. Let $F'$ be the fighting fish constructed from $T'$ using Algorithm \ref{algo:tree-to-fish}. Then the stem cell $w$ (resp. $x$, $y$ and $z$) is on the $E$ (resp. $N$, $W$ and $S$) side of the root cell $t'$ in $F'$.
By \ref{enu:B2}--\ref{enu:B5}, the tree $T'$ is exactly obtained from $T$ by making the exchanges of labels $N\leftrightarrow E$ and $S\leftrightarrow W$.
That is, if $v$ is a child of a non-root node $u$ in $T$ with $\lambda(v)=N$ (resp. $E$, $S$ and $W$) then the label of $v$ becomes $\lambda(v)=E$ (resp. $N$, $W$ and $S$) in $T'$ for all $\lambda(u)\in\{N,E,S,W\}$. It follows that if the cell $v$ is on the $N$ (resp. $E$, $S$ and $W$) side of its parent cell $u$ in $F$ then the cell $v$ becomes on the $E$ (resp. $N$, $W$ and $S$) side of $u$ in $F'$. Thus, $F'$ is the conjugation of $F$. Since $T'$ is associated to the tree $(t;T_2,T_1,T_4,T_3)\in\QT_n$, by the map $\varphi$ we have $\varphi:(t;T_2,T_1,T_4,T_3)\mapsto (F',t')=(F^{\cc},t^{\cc})$.

On the other hand, given a pair $(F,t)\in\CF_n$, let $T_{(F,t)}$ be the labeled rooted tree on the stem cells of $F$ obtained by a propagation of labels from the cell $t$ using Algorithm \ref{algo:fish-to-tree} with an arbitrary choice $\lambda(t)\in\{N,E,S,W\}$. Note that the tree $T_{(F,t)}$ satisfies conditions (i) and (ii) of Lemma \ref{lem:NESW-fish-tree}. By Lemma \ref{lem:NESW-representation},  $T_{(F,t)}$ is associated to a tree in $\QT_n$. Let $w$, $x$, $y$ and $z$ be the children of the root $t$ with label $N$, $E$, $S$ and $W$, respectively (if any). Let $T_w$ (resp. $T_x$, $T_y$ and $T_z$) be the subtree rooted at $w$ (resp. $x$, $y$ and $z$) of $T_{(F,t)}$. We define the a map $\varphi'$ by setting $\varphi':(F,t)\mapsto (t;T_1,T_2,T_3,T_4)\in\QT_n$, where $T_1={\chi_N}^{-1}(T_w)$, $T_2={\chi_E}^{-1}(T_x)$, $T_3={\chi_S}^{-1}(T_y)$ and $T_4={\chi_W}^{-1}(T_z)$. 

For the pair $(F,t)$, the stem cell $w$ (resp. $x$, $y$ and $z$) is on the $N$ (resp. $E$, $S$ and $W$) side of the root cell $t$. Then in the conjugation $F^{\cc}$ of $F$, the stem cell $w^{\cc}$ (resp. $x^{\cc}$, $y^{\cc}$ and $z^{\cc}$) becomes on the $E$ (resp. $N$, $W$ and $S$) side of the cell $t^{\cc}$. Let $T_{(F^{\cc},t^{\cc})}$ be the labeled rooted tree on the stem cells of $F^{\cc}$ obtained by a propagation of labels from the cell $t^{\cc}$ using Algorithm \ref{algo:fish-to-tree}.
By abuse of notation, let $t$ (resp. $w$, $x$, $y$ and $z$) denote the node corresponding to the stem cell $t^{\cc}$ (resp. $w^{\cc}$, $x^{\cc}$, $y^{\cc}$ and $z^{\cc}$) of $F^{\cc}$, and let $T_w$ (resp. $T_x$, $T_y$ and $T_z$) denote the subtree rooted at $w$ (resp. $x$, $y$ and $z$) of $T_{(F^{\cc},t^{\cc})}$. By the map $\varphi'$, we have $\varphi':(F^{\cc},t^{\cc})\mapsto (t;T'_1,T'_2,T'_3,T'_4)$, where $T'_1={\chi_N}^{-1}(T_x)$, $T'_2={\chi_E}^{-1}(T_w)$, $T'_3={\chi_S}^{-1}(T_z)$ and $T'_4={\chi_W}^{-1}(T_y)$. Since $T'_1=T_2$, $T'_2=T_1$, $T'_3=T_4$ and $T'_4=T_3$, we have $\varphi':(F^{\cc},t^{\cc})\mapsto (t;T_2,T_1,T_4,T_3)$.

For the map $\varphi:(t;T_1,T_2,T_3,T_4)\mapsto (F,t)$ defined above, by Lemma \ref{lem:cell-union}(ii) the stem cells of $F$ consist of the cells corresponding to the nodes of $T$ and carry the same labeling as in $T$, where $T$ is the labeled rooted tree associated to the tree $(t;T_1,T_2,T_3,T_4)$.  Thus, by a propagation of labels from the root cell $t$ using Algorithm \ref{algo:fish-to-tree}, we observe that $T$ is exactly the labeled rooted tree on the stem cells of $F$. It follows that $\varphi':(F,t)\mapsto (t;T_1,T_2,T_3,T_4)$. Thus, $\varphi'\circ\varphi$ is an identity function on $\QT_n$.

On the other hand, for the map $\varphi':(F,t)\mapsto (t;T_1,T_2,T_3,T_4)$ defined above, the labeled rooted tree $T_{(F,t)}$ on the stem cells of $F$ is associated to the tree $(t;T_1,T_2,T_3,T_4)$. By Proposition \ref{pro:isomorphic}, the labeled rooted tree associated to $(t;T_1,T_2,T_3,T_4)$ determines a unique fighting fish in $\F_n$ using Algorithm \ref{algo:tree-to-fish}. It follows that $\varphi:(t;T_1,T_2,T_3,T_4)\mapsto (F,t)$. Thus, $\varphi\circ\varphi'$ is an identity function on $\CF_n$. Hence the map $\varphi$ is a bijection, with $\varphi'$ the inverse of $\varphi$. The announced bijection in Theorem \ref{thm:fighting-fish-marked-stem-cell} is established. 
\qed

\medskip
For all $n\ge 1$, we define
\begin{equation}
\TF_n:=\{(F,t)\mid F\in\F_n\mbox{ and $t$ is a tail of $F$}\}.
\end{equation}
The bijection $\varphi$ in Theorem \ref{thm:fighting-fish-marked-stem-cell} restricted to the set $\TF_n$ results in the following bijection.

\begin{thm} \label{thm:pair-ternary-trees} For all $n\ge 1$, there is a bijection $\varphi^*:\TF_n\rightarrow (\T\times\T)_{n-1}$
such that $(F,t)\mapsto (T_1,T_2)$ implies $(F^{\cc}, t^{\cc})\mapsto (T_2,T_1)$.
\end{thm}

\begin{proof} Given a pair $(F,t)\in\TF_n$, by the bijection $\varphi$ in Theorem \ref{thm:fighting-fish-marked-stem-cell}, let $\varphi^{-1}:(F,t)\mapsto (t;T_1,T_2,T_3,T_4)$. Since $t$ is a tail of $F$, there is no cell on the $N$ or $E$ side of $t$. Thus, the subtrees $T_1$ and $T_2$ are empty. We define the map $\varphi^*$ by setting $\varphi^*:(F,t)\mapsto (T_3,T_4)\in(\T\times\T)_{n-1}$. Moreover, since $\varphi^{-1}:(F^{\cc},t^{\cc})\mapsto (t;T_2,T_1,T_4,T_3)$, it follows that $\varphi^*:(F^{\cc},t^{\cc})\mapsto (T_4,T_3)$.

The inverse map of $\varphi^*$ can be obtained as follows. Given a pair $(T_3,T_4)\in(\T\times\T)_{n-1}$, we form a tree $(t;T_1,T_2,T_3,T_4)\in\QT_n$, where $T_1$ and $T_2$ are empty. By Theorem \ref{thm:fighting-fish-marked-stem-cell}, let $\varphi:(t;T_1,T_2,T_3,T_4)\mapsto (F,t)$. Since $T_1$ and $T_2$ are empty, there is no stem cells on the $N$ or $E$ side of the root cell $t$ in $F$. Then $t$ is the top (resp. bottom) cell of the strip $\Asc(t)$ (resp. $\Des(t)$). Thus, $t$ is a tail of $F$. We define the map ${\varphi^*}^{-1}$ by setting ${\varphi^*}^{-1}:(T_3,T_4)\mapsto (F,t)\in\TF_n$.
Moreover, since $\varphi:(t;T_2,T_1,T_4,T_3)\mapsto (F^{\cc},t^{\cc})$, it follows that ${\varphi^*}^{-1}:(T_4,T_3)\mapsto (F^{\cc},t^{\cc})$. The announced bijection is established.
\end{proof}

\section{On fighting fish with a marked strip} \label{sec:marked-strip}

In this section, we shall establish the bijection $\phi:\T_n\rightarrow\MF_n$ in Theorem \ref{thm:TT-to-marked-fish} and the $(n+1)$-to-2 bijection between $\F_n$ and $\T_n$ in Theorem \ref{thm:2-to-(n+1)}.

\subsection{A proof of Theorem \ref{thm:TT-to-marked-fish}}

\begin{lem} \label{lem:the-map-phi}
For all $n\ge 1$, there is a bijection $\phi: T\mapsto (F,Q)$ of $\T_n$ onto $\MF_n$ such that the topmost stem cell of the descending strip $Q$ of the fighting fish $F$ is the cell associated to the root of the ternary tree $T$. 
\end{lem}

\begin{proof}
Given a ternary tree $T\in\T_n$, let $t$ be the root, and let $T_1$, $T_2$ and $T_3$ be the left, middle and right subtree of $t$, respectively (possibly empty). Then $T$ can be viewed as an ordered tree $(t; T_1, T_2, T_3, T_4)\in\QT_n$, where $T_4$ is empty. Let
$w$, $x$ and $y$ be the root of $T_1$, $T_2$ and $T_3$, respectively. By the map $\varphi$ in Theorem \ref{thm:fighting-fish-marked-stem-cell}, we obtain $\varphi:(t; T_1, T_2, T_3, T_4)\mapsto (F,t)$, where the stem cell $w$ (resp. $x$ and $y$) is on the $N$ (resp. $E$ and $S$) side of the root cell $t$ in the fighting fish $F$. Since there is no stem cell on the $W$ side of $t$, the root cell $t$ is the topmost stem cell of the strip $\Des(t)$. We define the map $\phi$ by setting $\phi:T\mapsto (F,\Des(t))\in\MF_n$.

On the other hand, given a pair $(F,Q)\in\MF_n$, let $t$ be the topmost stem cell of the descending strip $Q$. Then there is no stem cell on the $W$ side of $t$.
Let $w$, $x$ and $y$ be the stem cells connected to the cell $t$ on the $N$, $E$ and $S$ side of $t$, respectively (if any). 
By the map $\varphi^{-1}$ in Theorem \ref{thm:fighting-fish-marked-stem-cell}, we obtain $\varphi^{-1}:(F,t)\mapsto (t;T_1,T_2,T_3,T_4)$, where $T_1$, $T_2$ and $T_3$ are the ternary trees rooted at the nodes $w$, $x$ and $y$, respectively, and $T_4$ is empty.
Let $T\in\T_n$ be the ternary tree obtained from $(t;T_1,T_2,T_3,T_4)$ by setting $T_1$, $T_2$ and $T_3$ to be the left, middle and right subtrees of the root $t$. We define the map $\phi^{-1}$ by setting $\phi^{-1}:(F,Q)\mapsto T$.
The result follows.
\end{proof}

In the following, we focus on the vertex-labeling $\chi_E:\T_n\rightarrow\T^E_n$, so that the left, middle and right child of the root of a ternary tree are labeled by $N$, $E$ and $S$, respectively.

\begin{lem} \label{lem:v-labeling-rule} For any ternary tree $T\in\T_n$ and any node $u\in T$, under the bijection $\chi_E:\T_n\rightarrow\T^E_n$ the abscissa of $u$ is even if $\lambda(u)\in\{E,W\}$, and odd if $\lambda(u)\in\{N,S\}$.
\end{lem}

\begin{proof} Note that the root of $T$ is at abscissa 0 with label $E$. The result follows from the rules \ref{enu:B2}--\ref{enu:B5}.
\end{proof}

For convenience, each ternary tree $T\in\T_n$ is identified with the labeled rooted tree $\chi_E(T)\in\T^E_n$, so that  each node $u\in T$ carries the abscissa $\alpha(u)$ in $T$ and the label $\lambda(u)$ in $\chi_E(T)$.

\begin{lem}  \label{lem:unique-stem-cell-in-strips}
For any tree $T\in\T_n$, let $t$ be the root of $T$. Under the bijection $\phi$ in Lemma \ref{lem:the-map-phi}, let $(F,\Des(t))=\phi(T)\in\MF_n$. Then the following properties hold.
\begin{enumerate}
\item For any descending strip $Q\neq \Des(t)$ of $F$, there is a node $u$ with label $\lambda(u)\in\{N,S\}$ in $T$ such that the strip $Q$ is created by $u$ in the construction of $F$, and the cell $u$ is the unique stem cell with label $N$ or $S$ in the strip $Q$.
\item For any ascending strip $R$ of $F$, there is a node $u$ with label $\lambda(u)\in\{E,W\}$ in $T$ such that the strip $R$ is created by $u$ in the construction of $F$, and the cell $u$ is the unique stem cell with label $E$ or $W$ in the strip $R$.
\item If the tree $T$ contains $i$ nodes with odd abscissas and $j$ nodes with even abscissas then the fighting fish $F$ contains $i+1$ descending strips and $j$ ascending strips.
\end{enumerate}
\end{lem}

\begin{proof}
For any chain $T_1\subset  \cdots \subset T_n$ of subtrees of $T$, let $F_j$ be the fighting fish constructed from $T_j$ using Algorithm \ref{algo:tree-to-fish}, for $1\le j\le n$. For any descending strip $Q\neq \Des(t)$ of $F$, the strip $Q$ is created by an operation \ref{enu:D1} or \ref{enu:D4} in $F_i$ for some $i>1$. Let $u=T_i\setminus T_{i-1}$. We have $\lambda(u)\in\{N,S\}$. We shall prove by induction that the cell $u$ remains the unique stem cell with label $N$ or $S$ in the strip $Q$ of $F_j$ for all $j\ge i$.

For $k>i$, suppose the cell $u$ is the unique stem cell with label $N$ (resp. $S$) in the strip $Q$ of $F_{k-1}$. Let $v=T_k\setminus T_{k-1}$. If the stem cell $v$ is in the strip $Q$ of $F_k$ then it is on the $E$ or $W$ side of $u$. Then the node $v$ is a descendant of $u$ with $\lambda(v)\in\{E,W\}$ in $T_k$. Thus, the cell $u$ remains the unique stem cell with label $N$ (resp. $S$) in the strip $Q$ of $F_k$. The assertion (i) follows from induction. 

The root cell $t$ is the unique stem cell with label $E$ in the ascending strip $\Asc(t)$ since the other stem cells in $\Asc(t)$ are on the $N$ or $S$ side of $t$. By Algorithm \ref{algo:tree-to-fish},
for any ascending strip $R\neq \Asc(t)$ of $F$, the strip $R$ is created by an operation \ref{enu:D2} or \ref{enu:D3} in $F_i$ for some $i>1$. Let $u=T_i\setminus T_{i-1}$. We have $\lambda(u)\in\{E,W\}$. A symmetric argument of the above proof applies to the assertion that the cell $u$ is the unique stem cell with label $E$ or $W$ in the strip $R$ of $F_j$ for all $j\ge i$. The assertion (ii) follows.

For any non-root node $u\in T$, if the abscissa of $u$ is odd then by Lemma \ref{lem:v-labeling-rule} we have $\lambda(u)\in\{N,S\}$, and hence by (i) the cell $u$ is the unique stem cell with label $N$ or $S$ in the strip $\Des(u)$ of $F$. Moreover, if the abscissa of $u$ is even then $\lambda(u)\in\{E,W\}$, and hence by (ii) the cell $u$ is the unique stem cell with label $E$ or $W$ in the strip $\Asc(u)$ of $F$. 

For the root $t$ of $T$, we have $\alpha(t)=0$ and $\lambda(t)=E$. In the fighting fish $F$, the root cell $t$ is the unique stem cell with label $E$ in $\Asc(t)$. Since $t$ is the topmost stem cell in $\Des(t)$, the other stem cells of $\Des(t)$ are on the $E$ side of $t$. Thus, the strip $\Des(t)$ contains no stem cell with label $N$ or $S$.
It follows that the difference between the number of descending strips in $F$ and the number of nodes with odd abscissas in $T$ is 1, and that the number of ascending strips equals the number of nodes with even abscissas. The assertion (iii) follows.
\end{proof}

By Lemmas \ref{lem:the-map-phi} and \ref{lem:unique-stem-cell-in-strips}(iii), we have established the announced bijection in Theorem \ref{thm:TT-to-marked-fish}.

\begin{exa} {\rm
Consider the fighting fish $F$ shown in Figure \ref{fig:stem-cell}.  If the cell $w$ is set to be the root cell then $(F,\Des(w))$ maps to the ternary tree shown in Figure \ref{fig:ternary-tree}. Moreover, if the cell $h$ is set to be the root cell then $(F,\Des(h))$ maps to the ternary tree shown Figure \ref{fig:ternary-tree-2}. In either case, the corresponding ternary tree contains 4 nodes with odd abscissas and 5 nodes with even abscissas. Note that the fighting fish $F$ contains 5 descending strips and 5 ascending strips. 

For the case $n=4$, we compile a list of one-to-one correspondences between the members of $\MF_4$ and $\T_4$ in Figure \ref{fig:mapping}. 
}
\end{exa}

\begin{figure}[ht]
\begin{center}
\psfrag{a}[][][1]{$a$}
\psfrag{b}[][][1]{$b$}
\psfrag{c}[][][1]{$c$}
\psfrag{d}[][][1]{$d$}
\psfrag{e}[][][1]{$e$}
\psfrag{f}[][][1]{$f$}
\psfrag{g}[][][1]{$g$}
\psfrag{h}[][][1]{$h$}
\psfrag{w}[][][1]{$w$}
\includegraphics[width=3.2in]{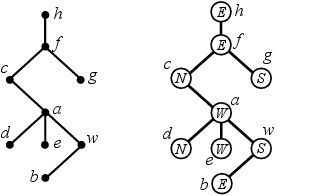}
\end{center}
\caption{\small A ternary tree and its vertex-labeling.} \label{fig:ternary-tree-2}
\end{figure}

\subsection{A combinatorial enumeration of fighting fish} 
Using the bijection $\phi$ in Theorem \ref{thm:TT-to-marked-fish}, we describe an $(n+1)$-to-2 bijection between the two sets $\F_n$ and $\T_n$.

For a strip $Q$ of a fighting fish, we use the notation $Q^{\cc}$ for the image of the strip $Q$ under a horizontal reflection. For $0\le \ell\le n-1$, define
\begin{equation} \label{eqn:def_T_n_ell}
\begin{aligned}
\T_{n,\ell} &= \{T\in\T_n\mid\mbox{$T$ contains $\ell$ nodes with odd abscissas}\}; \\
\F_{n,\ell} &= \{F\in\F_n\mid\mbox{$F$ contains $\ell+1$ descending strips}\}.
\end{aligned}
\end{equation} 

\medskip
\noindent
\emph{Proof of Theorem \ref{thm:2-to-(n+1)}.}
For any fighting fish $F\in\F_{n,\ell}$ ($0\le \ell\le n-1$), there are $\ell+1$ descending strips and $n-\ell$ ascending strips in $F$. Let $Q_1,\dots,Q_{\ell+1}$ be the descending strips of $F$. By Theorem \ref{thm:TT-to-marked-fish}, we have $\phi^{-1}(F,Q_j)\in\T_{n,\ell}$ for $j=1,\dots,\ell+1$. Hence
\begin{equation} \label{eqn:Q_is_descending}
(\ell+1)|\F_{n,\ell}| = |\T_{n,\ell}|.
\end{equation}
Let $R_1,\dots,R_{n-\ell}$ be the ascending strips of $F$. Then
$R^{\cc}_1,\dots,R^{\cc}_{n-\ell}$ are the descending strips of $F^{\cc}$.
By Theorem \ref{thm:TT-to-marked-fish},  we have $\phi^{-1}(F^{\cc},R^{\cc}_j)\in\T_{n,n-\ell-1}$ for $j=1,\dots, n-\ell$. Hence
\begin{equation} \label{eqn:Q_is_ascending}
(n-\ell)|\F_{n,\ell}| = |\T_{n,n-\ell-1}|.
\end{equation}
It follows that
\begin{align*}
(n+1)|\F_n| &= (n+1)\sum_{\ell=0}^{n-1} |\F_{n,\ell}| \\
            &= \sum_{\ell=0}^{n-1} \big(|\T_{n,\ell}|+|\T_{n,n-\ell-1}|\big) \\
           &= 2|\T_n|.
\end{align*}
We have established an $(n+1)$-to-2 bijection between $\F_n$ and $\T_n$. The result follows.
\qed

\section{Left ternary trees and fighting fish} \label{sec:left-ternary}

In this section, we describe the bijection $\psi:\LT_n\rightarrow\F_n$ in Theorem \ref{thm:LTT-to-fish-bijection}. First, we review the bijection $\phi:\T_n\rightarrow\MF_n$ in Theorem \ref{thm:TT-to-marked-fish} for some useful properties.

\subsection{Properties of the map $\phi$}

For any ternary tree $T\in\T_n$ and a node $x\in T$, let $L_x$ denote the subtree of $T$ consisting of $x$ and the left subtree of $x$. If $u$ is a descendant of $x$, let $\mu(x,u)$ denote the path from $x$ to $u$. Sometimes, the path $\mu(x,u)$ is expressed as a word on the alphabets $\{N,E,S,W\}$, which is formed by the labels of the nodes in the path. Let $P_x$ denote the parent of $x$. When $x$ is the root of $T$, the node $P_x$ can be viewed as a virtual node. Let $(F,Q)=\phi(T)\in\MF_n$. By Lemma \ref{lem:the-map-phi}, the root cell of $F$ is the topmost stem cell in the descending strip $Q$.
Note that if $\lambda(x)=E$ (resp. $N$) then in the fighting fish $F$ the stem cell $P_x$ is on the $W$ (resp. $S$) side of the cell $x$. When $x$ is the root of $T$, we have $\lambda(x)=E$ and the cell $P_x$ can be viewed as a virtual cell attached to the left upper edge of the strip $Q$. We say that an ascending strip \emph{intersects} a descending strip if these two strips have a cell in common.

\begin{pro} \label{pro:path} Let $T\in\T_n$, and let $(F,Q)=\phi(T)\in\MF_n$. Let $x, y\in T$ be two nodes with $\lambda(x)=E$ and $\lambda(y)=N$.  Then the following results hold.
\begin{enumerate}
\item If $a$ is a node in the left subtree of $x$ such that $\alpha(a)=\alpha(x)$ and every node $u\in\mu(x,a)$ satisfies the condition $\alpha(u)\ge \alpha(x)$ then $\lambda(a)=W$, $\Asc(a)$ intersects $\Des(x)$ at a cell between the stem cells $x$ and $P_x$, and the cell $a$ is the first stem cell in $\Asc(a)$ above $\Des(x)$ in $F$. 
\item If $b$ is a node in the left subtree of $y$ such that $\alpha(b)=\alpha(y)$ and every node $u\in\mu(y,b)$ satisfies the condition $\alpha(u)\ge \alpha(y)$ then $\lambda(b)=S$, $\Des(b)$ intersects $\Asc(y)$ at a cell between the stem cells $y$ and $P_y$, and the cell $b$ is the first stem cell in $\Des(b)$ below $\Asc(y)$ in $F$.
\end{enumerate}
\end{pro}

\begin{proof} Since $\lambda(x)=E$ and $\lambda(y)=N$, by Lemma \ref{lem:v-labeling-rule}, $\alpha(x)$ is even and $\alpha(y)$ is odd.  Let $\alpha(x)=2k$ and $\alpha(y)=2k+1$ for some $k\in\ZZ$. Let $\delta(x,a):= \max\{\alpha(u)-\alpha(x): u\in \mu(x,a)\}$ and $\delta(y,b):=\max\{\alpha(u)-\alpha(y): u\in \mu(y,b)\}$.
We shall prove the assertions (i) and (ii) by induction on $\delta(x,a)$ and $\delta(y,b)$.  

Since $\lambda(x)=E$ and the node $a$ is in the left subtree of $x$, if $\delta(x,a)=1$ then by \ref{enu:B2}--\ref{enu:B5} the path $\mu(x,a)$ is of the form 
\begin{equation} \label{eqn:mu(x,a)-1}
\mu(x,a): EN^{r_1}W^{s_1}N^{r_2}W^{s_2}\cdots N^{r_t}W^{s_t}
\end{equation}
for some strictly positive integers $r_1,s_1,r_2,s_2,\dots,r_t,s_t$ and $t$. Traversing the path $\mu(x,a)$ from $x$, we encounter the first segment $N^{r_1}W^{s_1}$, which consists of $r_1$ nodes $u_1, \dots,u_{r_1}$ at abscissa $2k+1$ with label $N$, followed by $s_1$ nodes $v_1,\dots,v_{s_1}$ at abscissa $2k$ with label $W$. Then in $F$ the cells $u_1,\dots,u_{r_1}$ are on the $N$ side of $x$, and the cells $v_1,\dots,v_{s_1}$ are on the $W$ side of $u_{r_1}$. Since $v_1$ a child of $u_{r_1}$ with $\lambda(v_1)=W$, the cell $v_1$ is created by an operation \ref{enu:D3} on $u_{r_1}$, and the strip $\Asc(v_1)$ is inserted to the left of $\Asc(x)$ intersecting $\Des(x)$ at a cell between $x$ and $P_x$. Similarly, since $v_i$ is a child of $v_{i-1}$ with $\lambda(v_i)=W$, the strip $\Asc(v_i)$ is inserted to the left of $\Asc(v_{i-1})$, for $2\le i\le s_1$. Thus, $\Asc(v_1),\dots,\Asc(v_{s_1})$ are distinct strips intersecting $\Des(x)$ at cells between $x$ and $P_x$. Note that the same observation holds for each segment $N^{r_i}W^{s_i}$ of $\mu(x,a)$. Thus, for the terminal node $a$ of $\mu(x,a)$, we have $\lambda(a)=W$, and the assertion (i) holds for this initial case.

Since $\lambda(y)=N$ and the node $b$ is in the left subtree of $y$, if $\delta(y,b)=1$ then by \ref{enu:B2}--\ref{enu:B5} the path $\mu(y,b)$ is of the form
\begin{equation} \label{eqn:mu(y,b)-1}
\mu(y,b): NE^{p_1}S^{q_1}E^{p_2}S^{q_2}\cdots E^{p_{t'}}S^{q_{t'}}
\end{equation}
for some strictly positive integers $p_1,q_1,p_2,q_2,\dots,p_{t'},q_{t'}$ and $t'$. Note that in $F$ the successive cells along the path $\mu(y,b)$ in (\ref{eqn:mu(y,b)-1}) is a conjugation of the case for $\mu(x,a)$ in (\ref{eqn:mu(x,a)-1}). A symmetric argument applies. Thus, we have $\lambda(b)=S$, and the assertion (ii) holds for this initial case.

Suppose the assertions (i) and (ii) hold for $\delta(x,a)\le m-1$ and  $\delta(y,b)\le m-1$, respectively. Let $\mu(x,a)$ be a path with $\delta(x,a)=m$. Since every node $u\in\mu(x,a)$ satisfies the condition $\alpha(u)\ge 2k$, there exist nodes $y_1, b_1, \dots, y_t, b_t$ at abscissa $2k+1$ such that the path $\mu(x,a)$ can be factorized as
\begin{equation} \label{eqn:factorization-i}
\mu(x,a): EN^{r_1-1}\mu(y_1,b_1)W^{s_1}N^{r_2-1}\mu(y_2,b_2)W^{s_2}\cdots N^{r_t-1}\mu(y_t,b_t)W^{s_t}
\end{equation}
for some strictly positive integers $r_1,s_1,r_2,s_2,\dots,r_t,s_t$ and $t$, where $\lambda(y_i)=N$, $b_i\in L_{y_i}$, $\alpha(u)\ge 2k+1$ for every node $u\in\mu(y_i,b_i)$, and $\delta(y_i,b_i)\le m-1$ for each $i$. Traversing the path $\mu(x,a)$ from $x$, for the first segment $N^{r_1-1}\mu(y_1,b_1)W^{s_1}$, we encounter $r_1$ nodes $u_1,\dots,u_{r_1}=y_1$ at abscissa $2k+1$ with label $N$, followed by $\mu(y_1,b_1)$ and then $s_1$ nodes $v_1,\dots,v_{s_1}$ at abscissa $2k$ with label $W$. If $y_1=b_1$ then the first segment is $N^{r_1}W^{s_1}$. By the above argument for (\ref{eqn:mu(x,a)-1}), the strips $\Asc(v_1),\dots,\Asc(v_{s_1})$ intersect $\Des(x)$ at cells between $x$ and $P_x$. 
Suppose $y_1\neq b_1$.
Note that the stem cells $u_1,\dots,u_{r_1}=y_1$ are on the $N$ side of $x$, so $\Asc(y_1)=\Asc(x)$. Since $\mu(y_1,b_1)\le m-1$, by induction hypothesis we have $\lambda(b_1)=S$,  $\Des(b_1)$ intersects $\Asc(y_1)$ at a cell between $y_1$ and $P_{y_1}=u_{r_1-1}$, and the cell $b_1$ is the first stem cell in $\Des(b_1)$ below $\Asc(y_1)$. Since $v_1$ is a child of $b_1$ with $\lambda(v_1)=W$, the cell $v_1$ is created by an operation \ref{enu:D3} on $b_1$, and the strip $\Asc(v_1)$ is inserted to the left of $\Asc(y_1)=\Asc(x)$ intersecting $\Des(x)$ at a cell between $x$ and $P_x$. Similarly, the strip $\Asc(v_i)$ is inserted to the left of $\Asc(v_{i-1})$, for $2\le i\le s_1$. Thus,  $\Asc(v_1),\dots,\Asc(v_{s_1})$ intersect $\Des(x)$ at cells between $x$ and $P_x$. The same observation holds for each segment $N^{r_i-1}\mu(y_i,b_i)W^{s_i}$ of the path $\mu(x,a)$. Thus, for the terminal node $a$ of $\mu(x,a)$, we have $\lambda(a)=W$, and the assertion (i) follows.

Let $\mu(y,b)$ be a path with $\delta(y,b)=m$. Since every node $u\in\mu(y,b)$ satisfies the condition $\alpha(u)\ge 2k+1$, there exist nodes $x_1, a_1, \dots, x_{t'}, a_{t'}$ at abscissa $2k+2$ such that the path $\mu(y,b)$ can be factorized as
\begin{equation} \label{eqn:factorization-ii}
\mu(y,b): NE^{p_1-1}\mu(x_1,a_1)S^{q_1}E^{p_2-1}\mu(x_2,a_2)S^{q_2}\cdots E^{p_{t'}-1}\mu(x_{t'},a_{t'})S^{q_{t'}},
\end{equation}
for some strictly positive integers $p_1,q_1,p_2,q_2,\dots,p_{t'},q_{t'}$ and $t'$, where $\lambda(x_i)=E$, $a_i\in L_{x_i}$, $\alpha(u)\ge 2k+2$ for every node $u\in\mu(x_i,a_i)$, and $\delta(x_i,a_i)\le m-1$ for each $i$. Note that in $F$ the successive cells along the path $\mu(y,b)$ in (\ref{eqn:factorization-ii}) is a conjugation of the case for $\mu(x,a)$ in (\ref{eqn:factorization-i}). A symmetric argument applies. Thus, we have $\lambda(b)=S$, and the assertion (ii) follows.
\end{proof}

\begin{exa} {\rm
Consider the fighting fish $F$ shown in Figure \ref{fig:stem-cell} and the descending strip $\Des(w)$ of $F$. The ternary tree $T=\phi^{-1}(F,\Des(w))$ is shown in Figure \ref{fig:ternary-tree}. Note that $\lambda(w)=E$ and $\alpha(w)=0$, and that $e$ (resp. $h$) is a node with abscissa 0 in $L_w$ such that every node $u\in\mu(w,e)$ (resp. $\mu(w,h)$) satisfies the condition $\alpha(u)\ge 0$. Notice that the strip $\Asc(e)$ (resp. $\Asc(h)$) intersects $\Des(w)$ at a cell between  $w$ and $P_w$ (virtual cell), and the cell $e$ (resp. $h$) is the first stem cell in $\Asc(e)$ (resp. $\Asc(h)$) above $\Des(w)$. 

Moreover, note that $\lambda(a)=N$ and $\alpha(a)=1$, and that $f$ (resp. $g$) is a node with abscissa 1 in $L_a$ such that every node $u\in\mu(a,f)$ (resp. $\mu(a,g)$) satisfies the condition $\alpha(u)\ge 1$. Notice that the strip $\Des(f)$ (resp. $\Des(g)$) intersects $\Asc(a)$ at a cell between the stem cells $a$ and $w$, and the cell $f$ (resp. $g$) is the first stem cell in $\Des(f)$ (resp. $\Des(g)$) below $\Asc(a)$.
}
\end{exa}

Let $T\in\T_n$, and let $(F,Q)=\phi(T)\in\MF_n$. For any node $x\in T$ with $\lambda(x)=E$, we define
\begin{equation*} 
\A(x) := \{R\mid \mbox{$R$ is an ascending strip of $F$ intersecting $\Des(x)$ at a cell between $x$ and $P_x$}\}.
\end{equation*}
Note that the strip $R$ is on the left of $\Asc(x)$ if $R\in\A(x)$ since $x$ is on the $E$ side of $P_x$.

\begin{lem} \label{lem:unique-W}
Let $T\in\T_n$, and let $(F,Q)=\phi(T)\in\MF_n$. Let $x\in T$ be a node with $\lambda(x)=E$. For any ascending strip $R$ of $F$, if $R\in\A(x)$ then there is a unique stem cell $q\in R$ with label $W$, and the node $q$ is in the left subtree of $x$ in $T$.
\end{lem}

\begin{proof} Choose a chain $T_1\subset \cdots\subset T_n$ of subtrees of $T$, and let $F_j$ be the fighting fish constructed from $T_j$ using Algorithm \ref{algo:tree-to-fish} for $1\le j\le n$. Let $x=T_i\setminus T_{i-1}$ for some $i$. For any ascending strip $R$ of $F$, by Lemma \ref{lem:unique-stem-cell-in-strips}(ii), the strip $R$ is created by a node, say $q$, with $\lambda(q)\in\{E,W\}$ in $T$, and the cell $q$ is a unique stem cell with label $E$ or $W$ in the strip $R$. 

Suppose the node $q$ is not a descendant of $x$, the cell $q$ can be arranged to be created before $x$ during the construction of $F$, i.e., $q=T_j\setminus T_{j-1}$ for some $j<i$. By Lemma \ref{lem:unique-stem-cell-in-strips}(i), the strip $\Des(x)$ is created by a node $p$, where either $p$ is the root or $p=T_k\setminus T_{k-1}$ is a node with $\lambda(p)\in\{N,S\}$ for some $k<i$. Since $\lambda(x)=E$, we observe that in the fighting fish $F_i$ the cell $x$ is added on the $E$ side of $p$,  attached to the right lower edge of the bottom cell of $\Des(p)$. Thus, $\A(x)$ is empty in $F_i$. If $u$ is a non-stem cell newly added in the strip $R$ in $F_{\ell}$ for some $\ell >i$ then $u$ is the intersection of $R$ and a descending strip which is created by a node $v=T_{\ell}\setminus T_{\ell-1}$ with label $S$. Then $u\in\Des(v)\neq\Des(p)$, and hence the strip $R$ does not intersect $\Des(p)$ at a cell between $x$ and $P_x$. Thus, $R\not\in\A(x)$.

Suppose $R\in\A(x)$. Then the node $q=T_j\setminus T_{j-1}$, for some $j>i$, is a descendant of $x$. We claim that the cell $q$ is labeled with $W$. If $q$ is labeled with $E$ then by \ref{enu:D2} the strip $R$ consists of the cell $q$ itself when it is created, and then grows on the $N$ and $S$ sides of $q$ during the construction of $F$. 
If $u$ is a non-stem cell added in $R$ then it is the intersection of $R$ and a descending strip which is created by a node $v=T_{\ell}\setminus T_{\ell-1}$, for some $\ell>j$, with label $S$. Then $u\in\Des(v)\neq\Des(x)$, and hence $R\not\in\A(x)$, which is a contradiction. Thus, the cell $q$ is labeled with $W$.

By \ref{enu:D3}, the cell $q$ is the top cell of $\Asc(q)$ when it is created, and hence $q$ is above the strip $\Des(x)$. Let $u$ be the cell at the intersection of $R$ and $\Des(x)$. Consider the following cases.

Case 1. Suppose the node $q$ is in the middle subtree of $x$.   Since $\lambda(x)=E$, by \ref{enu:B2} the middle child of $x$ is labeled with $E$. Consider the cell circuit in $F$ consisting of the cells along the path $\mu(x,q)$, followed by the cells from $q$ to $u$ in $R$ and the cells from $u$ to $x$ in $\Des(x)$.
Traversing the cell circuit from $x$, we start in the $E$ direction. If we draw a closed curve connecting the centers of consecutive cells along this circuit, we observe that the curve goes above the strip $\Des(x)$ to reach $q$ either from the right side of $x$ or from the other way around. In the former case, the right point of $x$ will be enclosed within a bounded region of this curve. By Lemma \ref{lem:enclosure},  the cell $x$ is not a stem cell, a contradiction. In the latter case, there exists a node $v\in\mu(x,q)$ such that the cell $v$ is one of the leftmost stem cells of this circuit. In this case, the right point of $v$ will be enclosed within a bounded region of this curve. This implies that $v$ is not a stem cell, also a contradiction. Thus, the node $q$ is not in the middle subtree of $x$.

Case 2. Suppose the node $q$ is in the right subtree of $x$. By \ref{enu:B2}, the right child of $x$ is labeled with $S$. Traversing the above-mentioned cell circuit from $x$, we start in the $S$ direction. Similarly, the closed curve, connecting the centers of consecutive cells along the circuit, goes above the strip $\Des(x)$ either from the right side of $x$ or from the other way around. In either case, there exists a stem cell such that its right point will be enclosed within a bounded region of the curve, which results in a contradiction. Thus, the node $q$ is not in the right subtree of $x$. 

By the above two cases, the node $q$ is in the left subtree of $x$. The result follows.
\end{proof}

\begin{lem} \label{lem:disjoint}
Let $T\in\T_n$, and let $(F,Q)=\phi(T)\in\MF_n$.  Let $x\in T$ be a node with $\lambda(x)=E$. If $a$ is a descendant of $x$ in $T$ with $\lambda(a)=E$ then $\A(x)\cap\A(a)=\emptyset$ in $F$.
\end{lem}

\begin{proof}  Choose a chain $T_1\subset \cdots\subset T_n$ of subtrees of $T$. 
Suppose $a$ is a descendant of $x$ with $\lambda(a)=E$, say $a=T_h\setminus T_{h-1}$ for some $h$. 
Suppose $\A(x)\cap\A(a)\neq\emptyset$. Let $\ell$ be the least integer such that there exists a node $q=T_{\ell}\setminus T_{\ell-1}$ creating an ascending strip in $\A(x)\cap\A(a)$ for some $\ell>h$. By Lemma \ref{lem:unique-W}, the cell $q$ is the unique stem cell with label $W$ in $\Asc(q)$. By \ref{enu:D3}, the strip $\Asc(q)$ is inserted on the immediate left of an ascending step, say $R$. Consider the following possibilities for the strip $R$.

Case 1.  $R=\Asc(x)$. Since $a$ is a descendant of $x$ with $\lambda(a)=E$, by a similar argument in the proof of Lemma \ref{lem:unique-W}, we observe that the strip $\Des(a)$ is created by a node $p$, where $p$ is the root or a node with $\lambda(p)\in\{N,S\}$, and the set
 $\A(a)$ is empty when the cell $a$ is created. If $u$ is a non-stem cell added in $R$ afterword then $u$ is the intersection of $R$ and a descending strip created by a node $v=T_j\setminus T_{j-1}$ with label $S$ for some $j>h$. Then $u\in\Des(v)\neq\Des(p)$, and hence $R\not\in\A(a)$. Since $\Asc(q)$ is on the immediate left of $R$, it follows that $\Asc(q)\not\in\A(a)$, a contradiction.

Case 2.  $R=\Asc(a)$. Since $a$ is a descendant of $x$ with $\lambda(a)=E$, by Lemma \ref{lem:unique-stem-cell-in-strips}(ii), there is no stem cell with label $W$ in $R$.
By Lemma \ref{lem:unique-W}, we have $R\not\in\A(x)$. Since $\Asc(q)$ in on the immediate left of $R$,  it follows that $\Asc(q)\not\in \A(x)$, a contradiction.

Case 3.  $R\in\A(x)\cap\A(a)$. By Lemma \ref{lem:unique-W}, there exists an integer $i<\ell$ such that the node in $T_i\setminus T_{i-1}$ creates the strip $R$ in $\A(x)\cap\A(a)$, a contradiction.

Thus, $\A(x)\cap\A(a)=\emptyset$ and the result follows. 
\end{proof}

\smallskip
\begin{lem} \label{lem:No_E}
Let $T\in\T_n$, and let $(F,Q)=\phi(T)\in\MF_n$.  Let $x\in T$ be a node with $\lambda(x)=E$.
Suppose $R$ is an ascending strip in $\A(x)$, let $q$ be the stem cell with label $W$ in $R$. Then there is no node $a\in\mu(x,q)$ with $\lambda(a)=E$ and $a\neq x$ such that $\alpha(a)=\alpha(q)$ in $T$.
\end{lem}

\begin{proof} 
Suppose $R\in\A(x)$. By Lemma \ref{lem:unique-W}, the node $q$ is in the left subtree of $x$. Suppose the path $\mu(x,q)$ contains nodes other than $x$ with label $E$ at abscissa $\alpha(q)$. Let $a\in\mu(x,q)$ be the last node from $x$ with $\lambda(a)=E$ and $\alpha(a)=\alpha(q)$. By \ref{enu:B2}, the middle child of $a$ is also a node with label $E$ at abscissa $\alpha(q)$, and hence the node $q$ is not in the middle subtree of $a$. 
We claim that every node $u\in\mu(a,q)$ satisfies the condition $\alpha(u)\geq \alpha(a)$. 

Suppose there exists a node $u\in\mu(a,q)$ with $\alpha(u)<\alpha(a)=\alpha(q)$. Since $\lambda(q)=W$, by Lemma \ref{lem:v-labeling-rule}, $\alpha(q)$ is even.
By \ref{enu:B2}--\ref{enu:B5}, the node $q$ is either the middle child of a node with label $W$ or the right child of a node with label $N$ or $S$. Moreover, by \ref{enu:B2}--\ref{enu:B5}, if a node is a left child then it is labeled with $E$ or $N$.
Then there exists a node $w\in\mu(u,q)$ with $\lambda(w)=N$ and $\alpha(w)=\alpha(q)+1$, which is odd. It follows that there exists a node $v\in\mu(u,w)$ with $\lambda(v)=E$ and $\alpha(v)=\alpha(q)$. This is against the condition that $a\in\mu(x,q)$ is the last node from $x$ with $\lambda(a)=E$ and $\alpha(a)=\alpha(q)$. Thus, every node $u\in\mu(a,q)$ satisfies the condition $\alpha(u)\geq \alpha(a)$, and the node $q$ is in the left subtree of $a$. By Proposition \ref{pro:path}(i), the strip $\Asc(q)$ intersects $\Des(a)$ at a cell between $a$ and $P_a$. Thus, $R=\Asc(q)\in\A(a)$. By Lemma \ref{lem:disjoint}, $\A(a)\cap\A(x)=\emptyset$. Then $R\not\in\A(x)$, a contradiction.  Thus, the path $\mu(x,q)$ contains no node other than $x$ with label $E$ at abscissa $\alpha(q)$. The result follows. 
\end{proof}

\smallskip
\begin{pro} \label{pro:alpha(q)=alpha(x)}
Let $T\in\T_n$, and let $(F,Q)=\phi(T)\in\MF_n$. Let $x\in T$ be a node with $\lambda(x)=E$. Suppose $R$ is an ascending strip in $\A(x)$, let $q$ be the stem cell with label $W$ in $R$. Then $\alpha(q)=\alpha(x)$ in $T$ and every node $u\in\mu(x,q)$ satisfies the condition $\alpha(u)\ge \alpha(x)$.
\end{pro}

\begin{proof} 
Suppose $R\in\A(x)$. By Lemma \ref{lem:unique-W}, the node $q$ is in the left subtree of $x$. Since $\lambda(x)=E$, by \ref{enu:B2} the left child of $x$ is labeled with $N$. Since $\lambda(q)=W$, by \ref{enu:D3} the cell $q$ is above the strip $\Des(x)$ since it is the top cell of $\Asc(q)$ when it is created. Let $u$ be the cell at the intersection of $R$ and $\Des(x)$. We claim that $\alpha(q)=\alpha(x)$. 

Suppose $\alpha(q)<\alpha(x)$. Let $b\in\mu(x,q)$ be the first node from $x$ with $\alpha(b)=\alpha(x)-1$, and let $a$ be the parent of $b$. Then $\alpha(a)=\alpha(x)$ and every node $u\in\mu(x,a)$ satisfies the condition $\alpha(u)\ge\alpha(x)$. By Proposition \ref{pro:path}(i), we have $\lambda(a)=W$, $\Asc(a)$ intersects $\Des(x)$ at a cell between $x$ and $P_x$, and $a$ is the first stem cell in $\Asc(a)$ above $\Des(x)$. Since $b$ is the right child of $a$, by \ref{enu:B4} we have $\lambda(b)=S$. By \ref{enu:D4}, the cell $b$ is added on the $S$ side of $a$, so that the strip $\Des(b)$ is inserted below $\Des(x)$. 

Traversing the path $\mu(x,q)$ from the cell $x$ in the fighting fish $F$, we start in the $N$ direction. If we draw a closed curve connecting the centers of consecutive cells along the path $\mu(x,q)$, followed by the cells from $q$ to $u$ in $R$ and the cells from $u$ to $x$ in $\Des(x)$, we observe that the curve goes below the strip $\Des(x)$ when visiting the cells from $a$ to $b$. Then the curve goes above the strip $\Des(x)$ to reach $q$ either from the right side of $x$ or from the other way around. By the argument in the proof of Case 1 of Lemma \ref{lem:unique-W}, in either case there exists a stem cell such that its right point  will be enclosed within a bounded region of this curve, which results in a contradiction. Thus, $\alpha(q)\geq\alpha(x)$.

Suppose $\alpha(q)>\alpha(x)$. Since $\lambda(x)=E$ and $\lambda(q)=W$, by Lemma \ref{lem:v-labeling-rule}, $\alpha(x)$ and $\alpha(q)$ are even. Let $\alpha(q)=\alpha(x)+2\ell$ for some $\ell\ge 1$. Since $\lambda(q)=W$, by \ref{enu:B2}--\ref{enu:B5} the node $q$ is either the middle child or the right child of a node, and hence there is a node $w\in\mu(x,q)$ with $\lambda(w)=N$ and $\alpha(w)=\alpha(x)+2\ell+1$. It follows that there exists a node $a\in\mu(x,w)$ with $\lambda(a)=E$ such that $\alpha(a)=\alpha(x)+2\ell=\alpha(q)$. By Lemma \ref{lem:No_E}, we have $R\not\in\A(x)$, a contradiction. Thus, we have $\alpha(q)=\alpha(x)$.

Suppose there exists a node $u\in\mu(x,q)$ with $\alpha(u)<\alpha(x)$. Since $\alpha(q)=\alpha(x)$ and $q$ is either the middle child or the right child of a node, there exists a node $w\in\mu(u,q)$ with $\lambda(w)=N$ and $\alpha(w)=\alpha(q)+1$. It follows that there is a node $a\in\mu(u,m)$ with $\lambda(a)=E$ and $\alpha(a)=\alpha(q)$. By Lemma \ref{lem:No_E}, we have $R\not\in\A(x)$, a contradiction. Thus, every node $u\in\mu(x,q)$ satisfies the condition $\alpha(u)\geq\alpha(x)$.  The result follows.
\end{proof}

\subsection{A proof of Theorem \ref{thm:LTT-to-fish-bijection}}

\begin{lem} \label{lem:stem-cells-in-jaw}
For any fighting fish $F\in\F_n$, let $x_1,x_2,\dots,x_d$ be the stem cells in the jaw of $F$ for some $d$, where $x_1$ is the topmost stem cell, and set $\lambda(x_1)=E$.  By the bijection $\phi$ in Theorem \ref{thm:TT-to-marked-fish}, let $T=\phi^{-1}(F,\Des(x_1))\in\T_n$.  Then the following properties hold.
\begin{enumerate}
\item The tree $T$ is rooted at $x_1$, and the nodes $x_1,x_2,\dots,x_d$ form a path such that $x_i$ is the middle child of $x_{i-1}$, for $2\le i\le d$. Hence $\alpha(x_i)=0$ and $\lambda(x_i)=E$ for all $i$.
\item The tree $T$ can be decomposed into the ternary trees $L_{x_1}, L_{x_2},\dots, L_{x_d}$.
\end{enumerate}
\end{lem}

\begin{proof}
Since $x_1,x_2,\dots,x_d$ are the stem cells in the jaw of $F$, $x_i$ is on the $E$ side of $x_{i-1}$ for $2\le i\le d$, and there is no stem cell on the $S$ side of each $x_i$. Since $\lambda(x_1)=E$, by \ref{enu:E2} we have $\lambda(x_i)=E$ for all $i$.
By Lemma \ref{lem:the-map-phi}, $T=\phi^{-1}(F,\Des(x_1))$ is a ternary tree rooted at $x_1$.
By \ref{enu:B2}, the node $x_i$ is the middle child of $x_{i-1}$ for $2\le i\le d$, and $x_i$ has no right child for all $i$.  Thus, $\alpha(x_1)=\cdots=\alpha(x_d)=0$ and $T$ can be decomposed into the subtrees $L_{x_1},\dots, L_{x_d}$. The results follow.
\end{proof}

Using the bijection $\phi:\T_n\rightarrow\MF_n$ in Theorem \ref{thm:TT-to-marked-fish}, we now describe the bijection $\psi:\LT_n\rightarrow\F_n$ in Theorem \ref{thm:LTT-to-fish-bijection}.

\smallskip
\noindent
\emph{Proof of Theorem \ref{thm:LTT-to-fish-bijection}.}
Given a left ternary tree $T\in\LT_n$,  let $x_1, x_2,\dots, x_d$ be the nodes of $T$ for some $d\geq 1$ such that $x_1$ is the root and $x_i$ is the middle child of $x_{i-1}$ for $2\le i\le d$. Hence $\alpha(x_i)=0$ for all $i$. Set $\lambda(x_1)=E$.   By \ref{enu:B2}, we have $\lambda(x_i)=E$ for all $i$. By the map $\phi$ in Theorem \ref{thm:TT-to-marked-fish}, let $(F,Q)=\phi(T)\in\MF_n$. By Lemma \ref{lem:the-map-phi}, the cell $x_1$ is the topmost stem cell in the descending strip $Q$ and the cell $x_i$ is on the $E$ side of $x_{i-1}$, for $2\le i\le d$. We claim that the strip $Q$ is the jaw of $F$.

Suppose $Q$ is not the jaw of $F$. Let $Q'$ be the descending strip immediately below $Q$. Since the node $x_1$ is the root of $T$, we construct the cells $x_1, x_2,\dots,x_d$ before the strip $Q'$ during the construction of $F$. Then the strip $Q'$ is created by a node, say $b$, with label $S$ using \ref{enu:D4}.  Consider the following cases of the ascending strip containing the cell $b$.

Case 1. $\Asc(b)=\Asc(x_i)$ for some $i$. Then the cell $b$ is on the $S$ side of $x_i$. Since $\lambda(x_i)=E$ and $\lambda(b)=S$, by \ref{enu:B2} the node $b$ is the right child of $x_i$ in $T$, and hence $\alpha(b)=-1$, which is against the condition that $T$ is a left ternary tree.

Case 2. $\Asc(b)\in\A(x_i)$ for some $i$. By Lemma \ref{lem:unique-W}, there is a unique stem cell, say $a$, with label $W$ in $\Asc(b)$, and the node $a$ is in the left subtree of $x_i$ in $T$.  By Proposition \ref{pro:alpha(q)=alpha(x)}, we have $\alpha(a)=\alpha(x_i)=0$ in $T$ and every node $u\in\mu(x_i,a)$ satisfies the condition $\alpha(u)\geq\alpha(x_i)$. By Proposition \ref{pro:path}(i), $a$ is the  first stem cell in $\Asc(a)=\Asc(b)$ above $\Des(x_i)=Q$. Hence the node $b$ is adjacent to $a$. Since $\lambda(a)=W$ and $\lambda(b)=S$, by \ref{enu:B3} the node $b$ is the right child of $a$, and hence $\alpha(b)=-1$, which is against the same condition.

Thus, the strip $Q$ is the jaw of $F$.  We define the map $\psi$ by setting $\psi:T\mapsto F$.

On the other hand, given a fighting fish $F\in\F_n$, let $Q$ be the jaw of $F$. By Theorem \ref{thm:TT-to-marked-fish}, let $T=\phi^{-1}(F,Q)\in\T_n$. We claim that $T$ is a left ternary tree.

Let $x_1,x_2,\dots,x_d$ be the stem cells in the jaw $Q$ for some $d$, where $x_1$ is the topmost stem cell, and set $\lambda(x_1)=E$. 
By Lemma \ref{lem:stem-cells-in-jaw}, we have $\alpha(x_i)=0$ and $\lambda(x_i)=E$ for all $i$, and $T$ can be decomposed into the subtrees $L_{x_1},\dots,L_{x_d}$. Suppose $T$ contains nodes with negative abscissas, say in the subtree $L_{x_i}$ for some $i$. Let $b$ be one of the nodes in $L_{x_i}$ with abscissa $-1$ such that every node $u\in\mu(x_i,b)$, $u\neq b$, satisfies the condition $\alpha(u)\ge 0$. Let $a$ be the parent of $b$. Then $\alpha(a)=0$, and $b$ is the right child of $a$.  By Proposition \ref{pro:path}(i), we have $\lambda(a)=W$,  $\Asc(a)$ intersects $\Des(x_i)$ at a cell between $x_i$ and $x_{i-1}$, and the cell $a$ is the first stem cell in $\Asc(a)$ above $\Des(x_i)$ in $F$. Since $\lambda(a)=W$, by \ref{enu:B4} we have $\lambda(b)=S$. By \ref{enu:D4}, the cell $b$ is added on the $S$ side of $a$, so that the strip $\Des(b)$ is inserted below $\Des(x_i)$. Note that $\Des(x_i)$ is the jaw $Q$ of $F$.  That there is a descending strip below the jaw of $F$ is a contradiction. Thus, all nodes of $T$ have nonnegative abscissas and hence $T$ is a left ternary tree. We define the map $\psi^{-1}$ by setting $\psi^{-1}:F\mapsto T$.

We claim that the number of nodes with abscissa 0 in $T$ equals the number of cells in the jaw of $F$. Decompose the tree $T$ into the subtrees $L_{x_1},\dots,L_{x_d}$. For each $i$, suppose $L_{x_i}$ contains $\ell_i$ nodes at abscissa 0.
Let $a_{i,1}, a_{i,2}, \dots, a_{i,\ell_i}$ be the nodes at abscissa 0 in $L_{x_i}$, where $x_i=a_{i,1}$. Note that for each $a_{i,j}$ ($2\le j\le \ell_i$), every node $u\in\mu(x_i, a_{i,j})$ satisfies the condition $\alpha(u)\ge 0$. By Proposition \ref{pro:path}(i), we have $\lambda(a_{i,j})=W$, and $\Asc(a_{i,j})$ intersects $\Des(x_i)$ in a cell between $x_i$ and $x_{i-1}$. By Lemma \ref{lem:unique-stem-cell-in-strips}(ii), these strips $\Asc(a_{i,2}),\dots,\Asc(a_{i,\ell_i})$ are distinct.
Thus, there are at least $\ell_i-1$ cells between $x_i$ and $x_{i-1}$ in the jaw of $F$. 

Suppose $R$ is an ascending strip in $\A(x_i)$. By Lemmas \ref{lem:unique-W}, there is a unique stem cell, say $q$, with label $W$ in $R$, and the node $q$ is in the left subtree of $x_i$. By Proposition \ref{pro:alpha(q)=alpha(x)}, we have $\alpha(q)=\alpha(x_i)$ and every node $u\in\mu(x_i,q)$ satisfies the condition $\alpha(u)\ge\alpha(x_i)$. Thus, the node $q$ is one of the nodes at abscissa 0 in $L_{x_i}$. Thus, there are exactly $\ell_i-1$ cells between $x_i$ and $x_{i-1}$ in the jaw of $F$. The announced bijection $\psi:\LT_n\rightarrow\F_n$ is established.
\qed

\medskip
For the case $n=4$, we compile a list of one-to-one correspondence between the members of $\F_4$ and $\LT_4$ in Figure \ref{fig:mapping}, where the ternary trees with the root $a$ are left ternary trees.

\begin{cor} The number of left ternary trees having $i+1$ nodes with even abscissas and $j$ nodes with odd abscissas is 
\begin{equation}
\frac{1}{(i+1)(j+1)}\binom{2i+j+1}{j}\binom{i+2j+1}{i}.
\end{equation}
\end{cor}

This result follows from the bijection in Theorem \ref{thm:LTT-to-fish-bijection} and a result by Duchi, Guerrini, Rinaldi and Schaeffer \cite[Theorem 2]{DGRS-B}. The same formula has been obtained in the enumeration of non-separable planar maps \cite{Brown-Tutte} and two-stack sortable permutations \cite{GW-1996, JS-1998}.

\section{Additional enumerative results} \label{sec:additional}
In this section, we present some enumerative results based on the bijections in Theorem \ref{thm:TT-to-marked-fish} and Theorem \ref{thm:LTT-to-fish-bijection}.

\subsection{Enumeration of symmetric fighting fish}
A fighting fish $F$ is \emph{symmetric} if $F^{\cc}=F$, where the axis of symmetry is the horizontal line passing through the center of the head. Let $\SF_n$ denote the set of symmetric fighting fish of size $n$.
Note that a symmetric fighting fish contains an odd number of stem cells since the terminal cell on the axis is either a tail or a branch cell. Thus, $\SF_n$ is empty if $n$ is even. We define
\begin{equation}
\OSF_{2n+1} :=\{F\in\SF_{2n+1}\mid \mbox{$F$ contains an odd number of tails} \}.
\end{equation}

\begin{thm} \label{thm:SFF_enumeration} For all $n\ge 0$, the following results hold.
\begin{enumerate}
\item There is a bijection between $\SF_{2n+1}$ and $(\T\times\T)_n$.
\item There is a bijection between $\OSF_{2n+1}$ and $\T_n$.
\end{enumerate}
\end{thm}

\begin{proof} (i) We describe a map $\rho:\SF_{2n+1} \rightarrow (\T\times\T)_n$ as follows. Given a fighting fish $F\in\SF_{2n+1}$, let $t$ be the terminal cell on the axis. Then we have $(F^{\cc},t^{\cc})=(F,t)$. By the bijection in Theorem \ref{thm:fighting-fish-marked-stem-cell}, we obtain $\varphi^{-1}:(F,t)\mapsto (t;T_1,T_2,T_3,T_4)=(t;T_2,T_1,T_4,T_3)\in \QT_{2n+1}$. It follows that $T_1=T_2$ and $T_3=T_4$. Consider the following two cases. 

Case 1. $F$ contains an odd number of tails. Then the cell $t$ is a tail (see Figure \ref{fig:symmetric-FF}(i)), and hence there exists no cell on the $N$ (resp. $E$) side of $t$. It follows that the subtrees $T_1$ and $T_2$ are empty. Thus, we define the map $\rho:(F,t)\mapsto (T_1,T_3)\in (\T\times\T)_n$, where $T_1$ is empty.

Case 2. $F$ contains an even number of tails. Then the  cell $t$ is a branch cell (see Figure \ref{fig:symmetric-FF}(ii)), and hence there exists a stem cell on the $N$ (resp. $E)$ side of $t$ in the strip $\Asc(t)$ (resp. $\Des(t)$). It follows that $T_1$ and $T_2$ are non-empty. Thus, we defined the map $\rho:(F,t)\mapsto (T_1,T_3)\in (\T\times\T)_n$, where $T_1$ is non-empty.

On the other hand, given a pair $(T_1,T_2)\in(\T\times\T)_n$, we associate $(T_1,T_2)$ with an ordered tree $(t;T_1,T_1,T_2,T_2)\in\QT_{2n+1}$. By  Theorem \ref{thm:fighting-fish-marked-stem-cell}, we obtain $\varphi:(t;T_1,T_1,T_2,T_2)\mapsto (F,t)=(F^{\cc},t^{\cc})\in\CF_{2n+1}$. It follows that $F$ is symmetric and the root cell $t$ is the terminal cell on the axis. We defined the map $\rho^{-1}: (T_1,T_2)\mapsto F\in\SF_{2n+1}$. The assertion (i) follows.

(ii) Following the above mentioned Case 1, the bijection between $\OSF_{2n+1}$ and $\T_n$ can be obtained by restricting the map $\rho$ to the set $\OSF_{2n+1}$. The assertion (ii) follows.
\end{proof}

\begin{figure}[ht]
\begin{center}
\psfrag{t}[][][1]{$t$}
\psfrag{w}[][][1]{$w$}
\psfrag{w'}[][][1]{$w'$}
\includegraphics[width=4.2in]{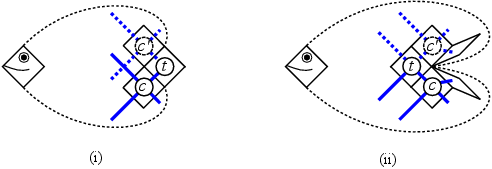}
\end{center}
\caption{\small The terminal cell on the axis of a symmetric fighting fish.} \label{fig:symmetric-FF}
\end{figure}

By (\ref{eqn:ternary_tree_number}), (\ref{eqn:ordered-pair-TT}) and Theorem \ref{thm:SFF_enumeration}, we have the following results. 

\begin{cor} \label{cor:symmetric-FF-enumeration}
For all $n\ge 0$, the following results hold.
\begin{enumerate}
\item The number of symmetric fighting fish of size $2n+1$ is $\frac{1}{n+1}\binom{3n+1}{n}$.
\item The number of symmetric fighting fish of size $2n+1$ containing an odd number of tails is $\frac{1}{2n+1}\binom{3n}{n}$.
\item The number of symmetric fighting fish of size $2n+1$ containing even number of tails is $\frac{1}{2n+1}\binom{3n}{n+1}$.
\end{enumerate}
\end{cor}

\subsection{Enumeration of left ternary trees}
We give a combinatorial proof of the fact that the number of left ternary trees with $n$ nodes is $\frac{2}{n+1}$ times the number of ternary trees with $n$ nodes. Define
\begin{align*}
\LT_n^{\odd} &= \{(T,u)\mid u\in T\in\LT_n\mbox{ and $\alpha(u)$ is odd}\}; \\
\LT_n^{\even} &=\{(T,u)\mid u\in T\in\LT_n\mbox{ and $\alpha(u)$ is even}\}.
\end{align*}
Using the bijections in Theorem \ref{thm:TT-to-marked-fish} and Theorem \ref{thm:LTT-to-fish-bijection}, we present a simplified proof of a result by Jacquard and Schaeffer (cf. \cite[Theorem 2]{JS-1998}). For $0\le \ell\le n-1$, define
\begin{equation}
\LT_{n,\ell} = \{T\in\LT_n\mid\mbox{$T$ contains $\ell$ nodes with odd abscissas}\}.
\end{equation}

\begin{thm} \label{thm:enumeration}
For $n\ge 1$, there exist bijections
\begin{align}
\LT_n^{\odd} & \rightarrow \T_n\setminus \LT_n,  \label{eqn:LT_odd}\\
\LT_n^{\even}& \rightarrow \T_n. \label{eqn:LT_even}
\end{align}
Hence we have 
\begin{equation*}
|\LT_n| =\frac{2}{n+1} |\T_n|.
\end{equation*}
\end{thm}

\begin{proof}
For $0\le \ell\le n-1$, we shall establish a bijection
\begin{equation} \label{eqn:LT^odd-bijection}
\{(T,u)\mid u\in T\in\LT_{n,\ell}\mbox{ and $\alpha(u)$ is odd}\}\rightarrow\T_{n,\ell}\setminus\LT_{n,\ell}.
\end{equation}

Given a left ternary tree $T\in\LT_{n,\ell}$, let $t$ be the root of $T$ with $\lambda(t)=E$, and let $u_1,\dots,u_{\ell}$ be the nodes with odd abscissas of $T$.  By Lemma \ref{lem:v-labeling-rule}, we have $\lambda(u_i)\in\{N,S\}$ for each $i$. By the map $\psi$ in Theorem \ref{thm:LTT-to-fish-bijection}, let $F=\psi(T)\in\F_n$.  Then $F$ contains $\ell+1$ descending strips. By Lemma \ref{lem:unique-stem-cell-in-strips}(i), $\Des(u_1),\dots,\Des(u_{\ell})$ are the $\ell$ descending strips other than the jaw of $F$. By the proof of Theorem \ref{thm:LTT-to-fish-bijection}, the strip $\Des(t)$ is the jaw of $F$.
By the bijection in Theorem \ref{thm:TT-to-marked-fish}, let $T_j=\phi^{-1}(F,\Des(u_j))$ for $j=1,\dots, \ell$. Then $T_j\in\T_n\setminus\LT_n$ containing $\ell$ nodes with odd abscissas. Then we obtain the map $(T,u_j)\mapsto T_j\in\T_{n,\ell}\setminus\LT_{n,\ell}$, for $j=1,\dots, \ell$.

On the other hand, given a ternary tree $T\in\T_{n,\ell}\setminus\LT_{n,\ell}$, let $(F,Q)=\phi(T)\in\MF_n$. By Theorem \ref{thm:TT-to-marked-fish}, the fighting fish $F$ contains $\ell+1$ descending strips. By Lemma \ref{lem:the-map-phi}, the root cell of $F$ is the topmost stem cell of the strip $Q$. Since $T$ is not a left ternary tree, by the proof of Theorem \ref{thm:LTT-to-fish-bijection}, $Q$ is not the jaw of $F$. Let $t$ be the topmost stem cell in the jaw of $F$, and let $T_F=\phi^{-1}(F,\Des(t))\in\T_n$. Then $T_F$ is a left ternary tree containing $\ell$ nodes with odd abscissas. By Lemma \ref{lem:unique-stem-cell-in-strips}(i), there is a unique stem cell, say $u$, with label $N$ or $S$ in the strip $Q$ of $F$. By Lemma \ref{lem:v-labeling-rule}, the node $u$ is at odd abscissa in $T_F$. Thus, we have the inverse map $T\mapsto (T_F,u)$. The bijection (\ref{eqn:LT^odd-bijection}) is established. Note that 
\begin{align*}
\LT_n^{\odd} &= \bigcup_{\ell=0}^{n-1} 
\{(T,u)\mid u\in T\in\LT_{n,\ell}\mbox{ and $\alpha(u)$ is odd}\}, \\
\T_n\setminus \LT_n &= \bigcup_{\ell=0}^{n-1} \big(\T_{n,\ell} \setminus\LT_{n,\ell}\big).
\end{align*}
The bijection in (\ref{eqn:LT_odd}) is established.

For $0\le \ell\le n-1$, we shall establish a bijection
\begin{equation} \label{eqn:LT^even-bijection}
\{(T,u)\mid u\in T\in\LT_{n,\ell}\mbox{ and $\alpha(u)$ is even}\}\rightarrow\T_{n,n-\ell-1}.
\end{equation}

Given a left ternary tree $T\in\LT_{n,\ell}$, let $v_1,v_2,\dots,v_{n-\ell}$ be the nodes with even abscissas of $T$, where $v_1$ is the root. By Lemma \ref{lem:v-labeling-rule}, we have $\lambda(v_i)\in\{E,W\}$ for each $i$. By the map $\psi$ in Theorem \ref{thm:LTT-to-fish-bijection}, let $F=\psi(T)\in\F_n$. Then $F$ contains $n-\ell$ ascending strips.
For $1\le j\le n-\ell$, let $R_j=\Asc(v_j)$.
By Lemma \ref{lem:unique-stem-cell-in-strips}(ii), $R_1,\dots,R_{n-\ell}$ are the $n-\ell$ ascending strips of $F$. Consider the conjugation of $F$. Note that $R^{\cc}_1,\dots,R^{\cc}_{n-\ell}$ become the $n-\ell$ descending strips of $F^{\cc}$.  By Theorem \ref{thm:TT-to-marked-fish}, let $T_j=\phi^{-1}(F^{\cc},R^{\cc}_j)$ for $1\le j\le n-\ell$. Then each $T_j$ contains $n-\ell-1$ nodes with odd abscissas. Then we obtain the map $(T,v_j)\mapsto T_j\in\T_{n,n-\ell-1}$, for $j=1,\dots, n-\ell$.

On the other hand, given a ternary tree $T\in\T_{n,n-\ell-1}$, let $(F,Q)=\phi(T)\in\MF_n$. By Theorem \ref{thm:TT-to-marked-fish}, the fighting fish $F$ contains $n-\ell$ descending strips. In the conjugation  $F^{\cc}$, the strip $Q^{\cc}$ becomes one of the $n-\ell$ ascending strips.  By Theorem \ref{thm:LTT-to-fish-bijection}, let $T_{F^{\cc}}=\psi^{-1}(F^{\cc})\in\LT_n$. Then the left ternary tree $T_{F^{\cc}}$ contains $n-\ell$ nodes with even abscissas and hence $\ell$ nodes with odd abscissas. By Lemma \ref{lem:unique-stem-cell-in-strips}(ii), there is a unique stem cell, say $u$, with label $E$ or $W$ in the ascending strip $Q^{\cc}$ of $F^{\cc}$. By Lemma \ref{lem:v-labeling-rule}, the node $u$ is at even abscissa in $T_{F^{\cc}}$. Thus, we have the inverse map $T\mapsto (T_{F^{\cc}},u)$. The bijection (\ref{eqn:LT^even-bijection}) is established. Note that
\begin{align*}
\LT_n^{\even} &= \bigcup_{\ell=0}^{n-1} 
\{(T,u)\mid u\in T\in\LT_{n,\ell}\mbox{ and $\alpha(u)$ is even}\}, \\
\T_n &= \bigcup_{\ell=0}^{n-1} \T_{n,n-\ell-1}.
\end{align*}
The bijection in (\ref{eqn:LT_even}) is established. The result follows.
\end{proof}

\section{Concluding Remarks} \label{sec:conclusion}
From a fresh perspective on the construction of fighting fish, we establish a bijection $\phi$ between the set $\T_n$ of ternary trees with $n$ nodes and the set of fighting fish of size $n$ with a marked descending strip. Using this result, we obtain a combinatorial enumeration of the set $\F_n$ of fighting fish of size $n$ by establishing an $(n+1)$-to-2 bijection between $\F_n$ and $\T_n$.

When the map $\phi$ is restricted to left ternary trees, we obtain a direct bijection $\psi$ between left ternary trees  and fighting fish. As remarked earlier, this bijection is not isomorphic to the composition of the bijection $\Phi$ between non-separable rooted planar maps and fighting fish in \cite{DH-2022-A} and the recursive bijection $\Psi$ between left ternary trees and non-separable rooted planar maps in \cite{DLDRP} or \cite{JS-1998}. This may be due to the convention on our choice of root cell.
For example, the left ternary tree in Figure \ref{fig:composition}(i) is carried by the map $\Psi$ to the non-separable rooted planar map in Figure \ref{fig:composition}(ii), which is carried by the map $\Phi$ to the fighting fish in Figure \ref{fig:composition}(iii). As shown in Figure \ref{fig:my-map}, this fighting fish is mapped to a different left ternary tree under our bijection $\phi$.

\begin{figure}[ht]
\begin{center}
\psfrag{Psi}[][][1]{$\Psi$}
\psfrag{Phi}[][][1]{$\Phi$}
\includegraphics[width=3.6in]{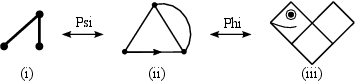}
\end{center}
\caption{\small (i) Left ternary tree, (ii) non-separable rooted planar map, and (iii) fighting fish.} \label{fig:composition}
\end{figure}

\begin{figure}[ht]
\begin{center}
\psfrag{a}[][][1]{$a$}
\psfrag{b}[][][1]{$b$}
\psfrag{c}[][][1]{$c$}
\psfrag{phi}[][][1]{$\psi$}
\includegraphics[width=2in]{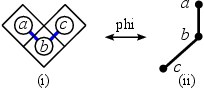}
\end{center}
\caption{\small A fighting fish and its corresponding  left ternary tree.} \label{fig:my-map}
\end{figure}

However, there is a refined conjecture that involves more statistics regarding the correspondence between left ternary trees  and fighting fish \cite{DGRS-B}. The \emph{fin} of a fighting fish is the path starting from the left point of its head, following its border counterclockwise, and ending at the right point of the first tail the path reaches. On the other hand, the \emph{core} of a ternary tree is the largest subtree including the root vertex and consisting only of left and middle edges. So far the following conjecture has still not been resolved.

\begin{con}[Duchi--Guerrini--Rinaldi--Schaeffer \cite{DGRS-B}]
The number of fighting fish with size $n$, fin length $k$, having $h$ tails, $i$ ascending strips and $j$ descending strips is equal to the number of left ternary trees with $n$ nodes, core size $k$, having $h$ right branches, $i+1$ non-root nodes with even abscissas and $j$ nodes with odd abscissas.
\end{con}

The following $q$-polynomial $F_n(q)$ is a natural $q$-analogue of the number of fighting fish of size $n$
\begin{equation*}
F_n(q):=\frac{[2]}{[n+1][2n+1]}{3n\brack n},
\end{equation*}
where $[k]=1+q+\cdots+q^{k-1}$, $[k]!=[1][2]\cdots [k]$, and ${n\brack k}=[n]!/([k]![n-k]!)$.
The initial $q$-polynomials are
\begin{align*}
F_1(q) &=1, \\
F_2(q) &=1+q^3,\\
F_3(q) &=1+q^3+q^4+q^6+q^7+q^{10}, \\
F_4(q) &=1+q^3+q^4+q^5+q^6+q^7+2q^8+2q^9+q^{10}+q^{11}+2q^{12}+2q^{13}+q^{14}+q^{15}+q^{16} \\
 &\qquad +q^{17}+q^{18}+q^{21}.
\end{align*}
These $q$-polynomials specialize at $q=-1$ to the aerated sequence 1, 0, 2, 0, 7, 0, 30, 0, 143, 0, 728, $\dots$, 
in which the terms alternate between 0 and the number of symmetric fighting fish. This $(-1)$-evaluation phenomenon is known in terms of synchronized intervals of Tamari lattices by Fang, Fusy and Nadeau \cite[Section 6.4]{FFN-2025}, where symmetric fighting fish correspond to synchronised self-dual intervals. It remains interesting to find a natural statistic directly on fighting fish  that describes the above distribution and $(-1)$-evaluation phenomenon.

\begin{figure}[ht]
\begin{center}
\psfrag{a}[][][0.75]{$a$}
\psfrag{b}[][][0.75]{$b$}
\psfrag{c}[][][0.75]{$c$}
\psfrag{d}[][][0.75]{$d$}
\includegraphics[width=6in]{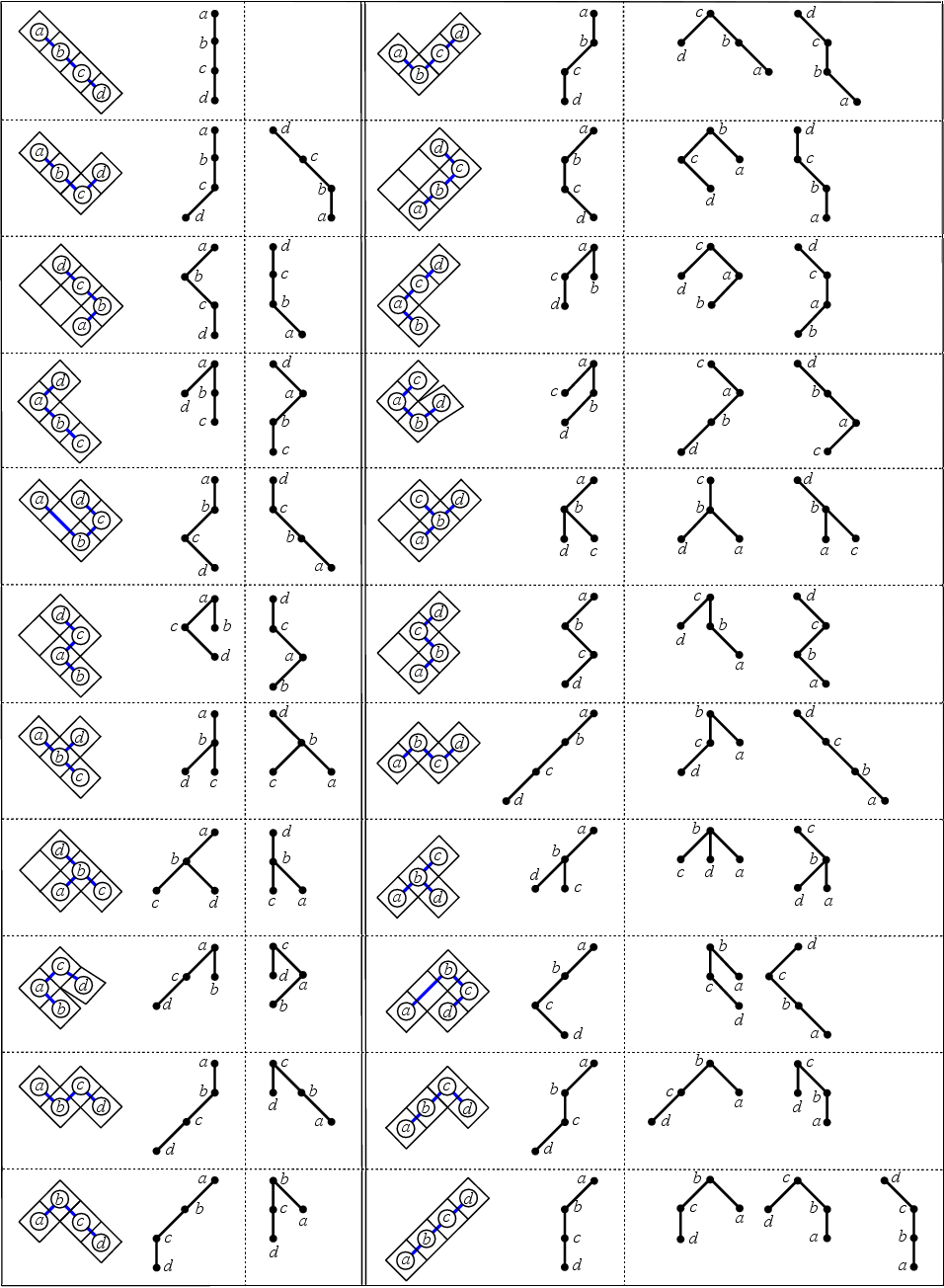}
\end{center}
\caption{\small The mapping of the set $\T_4$ (resp. $\LT_4$) onto $\F_4$.} \label{fig:mapping}
\end{figure}

\end{document}